\documentclass[letterpaper,10pt]{amsart}
\usepackage{indentfirst} 
\usepackage{amssymb}
\usepackage{amsmath}
\usepackage{mathabx} 
\usepackage{amsthm}
\usepackage{thmtools}
\usepackage{bbm}             
\usepackage[colorlinks=true]{hyperref}  
\usepackage[usenames,dvipsnames]{xcolor} 
\usepackage{tikz}
\usepackage{comment}

\newtheorem{theorem}{Theorem}[section]
\newtheorem{definition}[theorem]{Definition}
\newtheorem{proposition}[theorem]{Proposition}
\newtheorem{corollary}[theorem]{Corollary}
\newtheorem{lemma}[theorem]{Lemma}

\newtheorem{remark}[theorem]{Remark}

\newtheorem*{theorem-non}{Theorem}

\declaretheorem[name=Acknowledgements,numbered=no]{ack}

\theoremstyle{definition}
\newtheorem{example}[theorem]{Example}

\newcommand{\M}{\mathcal{M}}
\newcommand{\R}{\mathbb{R}}
\newcommand{\Z}{\mathbb{Z}}
\newcommand{\N}{\mathbb{N}}
\newcommand{\E}{\mathbb{E}}
\renewcommand{\P}{\mathbb{P}}

\newcommand{\cK}{\mathcal{K}}

\def\phi{\varphi}
\def\R{{\mathbb R}}

\def\N{{\mathbb N}}
\def\Z{{\mathbb Z}}
\def\disk{{\mathbb D}}

\def\E{{\mathcal E}}

\def\P{{\mathcal P}}
\def\Q{{\mathcal Q}}
\def\F{{\mathcal F}}

\def\M{{\mathcal M}}
\def\A{{\mathcal A}}

\def\le{\leqslant}
\def\ge{\geqslant}

\def\F{\mathcal{F}}
\def\M{\mathcal{M}}

\begin{document}

\title{Entropy in the cusp and phase transitions for geodesic flows}
\date{\today}


\author[G.~Iommi]{Godofredo Iommi}
\thanks{G.I. was partially  supported by  the Center of Dynamical Systems and Related Fields c\'odigo ACT1103 and by Proyecto Fondecyt 1150058.} \address{Facultad de Matem\'aticas,
Pontificia Universidad Cat\'olica de Chile (PUC), Avenida Vicu\~na Mackenna 4860, Santiago, Chile}
\email{giommi@mat.puc.cl}
\urladdr{\url{http://http://www.mat.uc.cl/~giommi/}}

\author[F. Riquelme]{Felipe Riquelme}
\thanks{F.R. was supported by Programa de Cooperaci\'on Cient\'ifica Internacional CONICYT-CNRS c\'odigo PCCI 14009.}  \address{IRMAR-UMR 6625 CNRS,
Universit\'e de Rennes 1, Rennes 35042, France}
\email{felipe.riquelme@univ-rennes1.fr}

\author[A. Velozo]{Anibal Velozo}  \address{Princeton University, Princeton NJ 08544-1000, USA.}
\email{avelozo@math.princeton.edu}

\begin{abstract}
In this paper we study the geodesic flow for a particular class of Riemannian non-compact manifolds with variable pinched negative sectional curvature. For a sequence of invariant measures we are able to prove results relating the loss of mass and bounds on the measure entropies. We  compute the entropy contribution of the cusps. We develop and study the corresponding thermodynamic formalism.  We obtain  certain regularity results for the pressure of a class of  potentials. We prove that the pressure is real analytic until it undergoes a phase transition, after which it becomes constant. Our techniques are based on the one side on symbolic methods and Markov partitions and on the other on geometric techniques and  approximation properties at level of groups.
\end{abstract}

\maketitle

\section{Introduction}

This paper is devoted to study thermodynamic formalism for a class of geodesic flows defined over non-compact manifolds of variable pinched negative curvature.  These flows can be coded with suspension flows defined over Markov shifts, albeit on a countable alphabet. This paper addresses two different but related problems where the non-compactness of the ambient manifold plays a fundamental role. Inspired in some recent results proved in the context of homogeneous dynamics (\cite{elmv,ekp}), we establish properties that relate  the escape of mass of a sequence of invariant probability measures for the geodesic flow with its measure theoretic entropies (see section \ref{e:m}). The second goal of the paper is to construct a class of potentials for which the pressure exhibits a phase transition (see section \ref{s:tf}). To obtain these results we combine both geometric and symbolic methods.

The class of manifolds that we will be working on along the paper were introduced in \cite{dalpei}.  These manifolds are obtained as the quotient of a Hadamard manifold with an extended Schottky group (see sub-section \ref{ss:sr} for precise definitions).  This class of groups have parabolic elements of rank 1, therefore the manifolds are non-compact. It was shown in \cite{dalpei} that the geodesic flow over the unit tangent bundle of those manifolds can be coded as suspension flows over countable Markov shifts. The existence of a Markov coding for the geodesic flow is essential for ours results.

The idea of coding a flow in order to describe its dynamical and ergodic properties has long history and a great deal of interesting and important results have been obtained with these methods. Probably, some of the most relevant results using this technique are related to counting closed geodesics and also estimating the rate  at which they grow \cite{papo}. A landmark result  is the construction of Markov partitions for Axiom A flows defined over compact manifolds done by Bowen \cite{bo1} and Ratner \cite{ra}. They actually showed that Axiom A flows can be coded with suspension flows defined over sub-shifts of finite type on finite alphabets with a regular (H\"older) roof functions. The study in the non-compact setting is far less developed. However, some interesting results have been obtained.
Recently, Hamenst\"adt \cite{h} and also Bufetov and Gurevich \cite{bg} have coded Teichm\"uller flows with suspension flows over countable alphabets and using this representation have proved, for example, the uniqueness of the measure of maximal entropy. Another important example for which codings on countable alphabets have been constructed is a type of Sinai billiards \cite{bs1, bs2}.

As mentioned before, one goal of the paper is to investigate the loss of mass of sequences of invariant measures for the geodesic flow. Recently, the loss of mass has been studied for the modular surface in \cite{elmv}. Despite  being a particular case, the method displayed in \cite{elmv} is quite flexible and has the advantage it can be understood purely geometrically. A more general situation is studied in \cite{ekp},  the results apply to any geodesic flow on a finite volume hyperbolic surface (with respect to the constant sectional curvature metric). We begin introducing the notion of \textit{entropy at infinity} of a dynamical system defined over a non-compact topological space.

\begin{definition} \label{def:ei}  Let $Y$ be a non-compact topological space and $S=(S_t)_{t\in\R}:Y\to Y$ a continuous  flow. We define the \emph{``entropy at infinity''} of the dynamical system as the number
$$h_\infty(Y,\{S_t\})=\sup_{\{ \nu_n \}\rightharpoonup 0}\limsup_{n\to\infty} h_{\nu_n}(S),$$
where the supremum is taken over all the sequences of invariant probability measures for the flow converging in the vague topology to the zero measure. If no such sequence exists we set $h_\infty(Y,\{S_t\})=0$. Here $h_\nu(S)$ denotes the measure-theoretic entropy of a probability $S$-invariant measure $\nu$.
\end{definition}
 Recall  that the total mass of probability measures is not necessarily preserved under vague convergence (as opposite to  weak convergence). Note that Definition \ref{def:ei} can be extended to more general group actions whenever an entropy theory has been developed for the group in consideration. Amenable groups are a classical example of such.

In this paper we are able to compute $h_\infty(T^1X/\Gamma,(g_t))$, where $X$ is a Hadamard manifold with pinched negative sectional curvature, $\Gamma$ is an extended Schottky group generated by $N_1$ hyperbolic isometries and $N_2$ parabolic ones, and $(g_t)$ is the geodesic flow on the unit tangent bundle $T^1X/\Gamma$ (see sub-section \ref{ss:sr} for precise definitions). Define
\begin{equation*}
\delta_{p,max}:=\max\{\delta_{p_{i}}, 1\leq i \leq N_{2}\},
\end{equation*}
where $\delta_{p_{i}}$ is the critical exponent of the Poincar\'e series of the group $<p_i>$ and $p_i$ is a parabolic element of $\Gamma$. We prove that
$$h_\infty(T^1X/\Gamma,(g_t))=\delta_{p,max}.$$
It worth mentioning $\delta_{p,max}$ is strictly less than the topological entropy of the geodesic flow. In our context, the non-compact pieces of our space are modeled by cusps. That is why we refer to this quantity as \textit{entropy in the cusps}.
More concretely we prove that if a sequence of measures are dissipating through the cusps, then the entropy contribution of the sequence is at most $\delta_{p,\max}$. In \cite{ekp} it is proven that $h_\infty(\Gamma \setminus G,A)=h_{top}/2$ where $G$ is a connected semisimple Lie group of real rank 1 with finite center, $\Gamma$ a lattice in $G$ and $A$ a one-parameter subgroup of diagonalizable elements over $\R$ acting by right multiplication. In particular $h_\infty(T^1S,(g_t))=1/2$, where $S$ is a hyperbolic surface with finite volume.
We also obtain results in the case where the sequence of measures keep some mass at the limit. Our bounds are less concrete than the analog result in the homogeneous dynamical case though.  The following is one of our main results and gives the calculation of the entropy in the cusps mentioned before. The manifold under consideration is assumed to satisfy some mild technical conditions, denoted by $(\star)$. This condition is explained in detail in subsection \ref{ss:sr}.

\begin{theorem}\label{thm:escape_mass} Let $X$ be a Hadamard manifold with pinched negative sectional curvature and let $\Gamma$ be an extended Schottky group of isometries of $X$ with property $(\star)$. Assume that the derivatives of the sectional curvature are uniformly bounded. Then, for every $c>\delta_{p,max}$ there exists a constant $m=m(c)>0$, with the following property: If $(\nu_n)$ is a sequence of $(g_t)$-invariant probability measures on $T^1X/\Gamma$ satisfying $h_{\nu_n}(g)\geq c$, then for every vague limit $\nu_n \rightharpoonup \nu$, we have
$$\nu(T^1X/\Gamma)\geq m.$$
In particular if $\nu_n \rightharpoonup 0$ then $\limsup h_{\nu_n}(g)\le \delta_{p,max}$. Moreover, the value $\delta_{p,max}$ is optimal in the following sense: there exists a sequence $(\nu_n)$ of $(g_t)$-invariant probability measures on $T^1X/\Gamma$ such that $h_{\nu_n}(g)\to \delta_{p, max}$ and $\nu_n \rightharpoonup 0$.
\end{theorem}

We believe similar results can be obtained for the geodesic flow in the geometrically finite case (maybe under some strong conditions on the curvature).

Our second goal is to study regularity properties of the pressure function. In order to do so, we make strong use of the symbolic coding
that the geodesic flow has in the manifolds we are considering.  The idea of using symbolic dynamics to study thermodynamic formalism of flows of geometric nature can be traced back to the work of  Bowen and Ruelle \cite{br}. They studied in great detail
ergodic theory and thermodynamic formalism for Axiom A flows defined on compact manifolds. The techniques they used were symbolic in nature and were based on the symbolic codings obtained by Bowen  \cite{bo1} and Ratner \cite{ra}. In this work we follow this strategy. We stress, however, that our symbolic models are non-compact. There are several difficulties  related to the lack of compactness that have to be addressed, but also new phenomena are observed.

To begin with,  in sub-section  \ref{s:tf} we propose a definition of topological pressure, $P(\cdot)$, that satisfies not only the variational principle, but also an approximation by compact invariant sets property. These provide symbolic proofs to results obtained by different (non-symbolic) methods in far more general settings by Paulin, Pollicott  and Schapira \cite{pps}.   The strength of our approach is perhaps better appreciated in our regularity results  for the pressure (sub-section \ref{s_phase}). Note that the techniques in \cite{pps} do not provide these type of results. We say that the pressure function $t \mapsto P(tf)$ has a \emph{phase transition} at $t=t_0$ if it is not analytic at that point. It readily follows from work by Bowen and Ruelle \cite{br} that the pressure for Axiom A flows and regular potentials is real analytic and hence has no phase transitions. Regularity properties of the pressure of geodesic flows defined on non-compact manifolds, as far as we know,  have not been studied, with the exception of the geodesic flow
defined on the modular surface (see \cite[Section 6]{ij}).

There is a general strategy used to study regularity properties of the pressure of maps and flows with strong hyperbolic or expanding properties in most of the phase space but not in all of it. Indeed, if there exists a subset of the phase space $B \subset X$ for which the restricted dynamics is not expansive and its entropy equal to $A$, then it is possible to construct potentials $f:X\to \R$ for which the pressure function has the form
\begin{equation} \label{e:1}
P(tf):=
\begin{cases}
\text{real analytic, strictly decreasing and convex} & \text{ if }  t < t';\\
A & \text{ if } t > t'.
\end{cases}
\end{equation}
Well known examples of this phenomena include the Manneville-Pomeau map (see for example \cite{Sar20}) in which the set $B$ consist of a parabolic fixed point and therefore $A=0$. The potential considered is the geometrical one: $-\log |T'|$.  Similar results for multimodal maps have been obtained, for example, in \cite{dt, it1, pr}. In this case the set $B$ corresponds to the post-critical set and $A=0$.  Examples of maps in which $A>0$ have been studied in \cite{dgr,it2}. For suspension flows over countable Markov shifts, similar examples were obtained in \cite{ij}. However, in that case the number  $A$, which we denote by $s_\infty$, remained unexplained for. In this paper we show that the entropy at infinity of a suspension flow over a countable Markov shift corresponds to $s_\infty$ (see Corollary \ref{cor:ent_inf}).

In the case of geodesic flows, roughly speaking we are considering the set $B$ as the union of the cusps of the manifold. More interestingly, as we mentioned before we are able to compute the entropy contributions of the cusps in the geodesic flow. In sub-section \ref{s_phase} we construct a class of potentials, that we denote by $ \mathcal{F}$,  for which the pressure exhibits similar behaviour as in equation \eqref{e:1}. In those examples $A=\delta_{p,max}$. Note that it is possible for $t'$ to be infinity  and in that case the pressure is real analytic. The following is the precise statement,

\begin{theorem} \label{thm:gpt}
Let $X$ be a Hadamard manifold with pinched negative sectional curvature and let $\Gamma$ be an extended Schottky group of isometries of $X$ with property $(\star)$. Assume that the derivatives of the sectional curvature are uniformly bounded. If $f \in \mathcal{F}$, then
\begin{enumerate}
\item For every $t \in \R$ we have that $P_g(tf) \geq \delta_{p,max}$.
\item We have that $\lim_{t \to -\infty} P_g(tf)= \delta_{p,max}$.
\item Let $t':= \sup \{t \in\R: P_g(tf) = \delta_{p,max}\}$, then
\begin{equation*}
P_g(tf)=
\begin{cases}
\delta_{p,max} & \text{ if } t <t'; \\
\text{real analytic, strictly convex, strictly increasing} & \text{ if } t>t'.
\end{cases}
\end{equation*}
\item If $t>t'$, the potential $tf$ has a unique equilibrium measure. If $t<t'$ it has no equilibrium measure.
\end{enumerate}
\end{theorem}

In order to prove this result we need to relate symbolic quantities with geometrical ones. This is achieved in Theorem \ref{s_inftyshottky}  in which a symbolic parameter of the suspension flow, the number $s_{\infty}$, is proven to be equal to the geometric parameter of the group $\delta_{p, max}$. We stress that when coding a flow a great deal of geometric information is lost. With this result we are able to recover part of it.

\bigskip
\begin{ack}
We thank B. Schapira for a careful reading of the paper and for several interesting and useful comments. We also thank F. Dal'bo and  F. Ledrappier   for several enlighten and  interesting discussions around the subject treated in this article. This work started when the second and third authors were visiting Pontificia Universidad Cat\'olica de Chile, they are very grateful for the great conditions and hospitality received in PUC Mathematics Department.
\end{ack}

\section{Preliminaries on thermodynamic formalism and suspension flows} \label{s:pre}
This section is devoted to provide the necessary background on thermodynamic formalism and on suspension flows required on the rest of the article.

\subsection{Thermodynamic formalism for countable Markov shifts}
Let $M$ be an incidence matrix defined on the alphabet of natural numbers. The associated one sided countable Markov shift $(\Sigma^+, \sigma)$
is the set
\[ \Sigma^+ := \left\{ (x_n)_{n \in \N} : M(x_n, x_{n+1})=1  \text{ for every } n \in \N   \right\}, \]
together with the shift map $\sigma: \Sigma^+  \to \Sigma^+ $ defined by
$\sigma(x_1, x_2, \dots)=(x_2, x_3,\dots)$. A standing assumption we will make throughout the article is that $(\Sigma^+, \sigma)$ is \emph{topologically mixing}. We equip $\Sigma^+$  with the topology generated by the cylinder sets
\begin{equation*}
C_{a_1 \cdots a_n}= \{ x\in \Sigma^+: x_i=a_i \ \text{for
$i=1,\ldots,n$}\}.
\end{equation*}
We stress that, in general, $\Sigma^+$ is a non-compact space. Given a function $ \phi \colon \Sigma^+ \to \R$ we define the
$n-$th variations of $\phi$ by
\[ V_{n}(\phi):= \sup \{| \phi(x)- \phi(y)| : x,y\in \Sigma^+, \ x_{i}=y_{i}
\ \text{for $i=1,\ldots,n$} \},
\]
where $x=(x_1 x_2 \cdots)$ and $y=(y_1y_2 \cdots)$.  We say that $\phi$ has \emph{summable variation} if $\sum_{n=1}^{\infty} V_n(\phi)<\infty$. We say that  $\phi$ is \emph{locally H\"older} if there exists $\theta \in (0,1)$ such that for all $n \geq 1$, we have  $V_{n}( \phi) \leq O( \theta^{n}). $

This section is devoted to recall some of the notions and results of thermodynamic formalism in this setting. The following definition was introduced by Sarig \cite{Sar99} based on work by Gurevich \cite{gu}.

\begin{definition} \label{presion}
Let $\phi \colon \Sigma^+ \to \R$ be a function of summable variation. The
\emph{Gurevich pressure} of $\phi$  is defined by
\[
 P(\phi) = \lim_{n \to
\infty} \frac{1}{n} \log \sum_{x:\sigma^{n}x=x} \exp \left(
\sum_{i=0}^{n-1} \phi(\sigma^{i}x)\right)  \chi_{C_{i_{1}}}(x),
\]
where $\chi_{C_{i_{1}}}(x)$ is the characteristic function of the
cylinder $C_{i_{1}} \subset \Sigma^+$.
\end{definition}
It is possible to show (see  \cite[Theorem 1]{Sar99}) that the limit always exists and that it does not depend on $i_1$. The following two properties of the pressure will be relevant for our purposes (see \cite[Theorems 2 and 3]{Sar99} and \cite[Theorem 2.10]{ijt}). If $\phi \colon \Sigma^+ \to \R$ is a function of summable variations, then
\begin{enumerate}
\item (Approximation property)
\begin{equation*}\label{*tan}
P ( \phi) = \sup \{ P_{K}( \phi) : K\in \cK \},
\end{equation*}
where $P_{K}( \phi)$ is the classical topological pressure  on
$K$ (see \cite[Chapter 9]{wa}) and
\[
\cK:= \{ K \subset \Sigma^+ : K \ne \emptyset \text{ compact and }
\sigma\text{-invariant}\}.
\]

\item (Variational Principle) Denote by $\M_{\sigma}$ is the space of $\sigma-$invariant probability measures and by $h_\mu(\sigma)$ the entropy of the measure $\mu$ (see \cite[Chapter 4]{wa}). If $\phi \colon \Sigma^+ \to \R$ is a function of summable variation then,
\[P_\sigma(\phi)= \sup \left\{ h_\mu(\sigma) + \int \phi \text{d} \mu : \mu \in \M_{\sigma} \text{ and } - \int \phi \text{d} \mu < \infty \right\}.\]

\end{enumerate}
A measure $\mu \in \M_{\sigma}$ attaining the supremum, that is, $P_\sigma(\phi)= h_\mu(\sigma) + \int \phi \text{d} \mu$ is called \emph{equilibrium measure} for $\phi$.  A potential of summable variations has at most one equilibrium measure (see \cite[Theorem 1.1]{busa}).

It turns out that under a combinatorial assumption on the incidence matrix $M$, which roughly means to be similar to a full-shift, the thermodynamic formalism is well behaved.

\begin{definition} \label{def:bip}
We say that a countable Markov shift $(\Sigma^+, \sigma)$, defined by the transition matrix $M(i,j)$ with $(i,j)\in \N  \times \N $, satisfies the \emph{BIP  (Big Images and Preimages) condition} if and only if  there exists $\{b_1 , \dots , b_n\}  \subset \N $ such that for every $a \in \N $ there exists $i,j \in \N$ with $M(b_i, a)M(a,b_j)=1$.
\end{definition}

The following theorem summarises results proven by Sarig in \cite{Sar99,Sar20} and by Mauldin and Urba\'{n}ski, \cite{mubook}, where they show that thermodynamic formalism in this setting is similar to that observed for sub-shifts of finite type on finite alphabets.

\begin{theorem} \label{bip}
Let $(\Sigma^+, \sigma)$ be a countable Markov shift satisfying the BIP condition and $\phi: \Sigma^+ \to \R$ a non-positive locally H\"older potential. Then, there exists $s_{\infty} >0$ such that pressure function $t \to P_\sigma(t\phi)$ has the following properties
\begin{equation*}
P_\sigma(t \phi)=
\begin{cases}
\infty  & \text{ if  } t  < s_{\infty}; \\
\text{real analytic } & \text{ if  } t > s_{\infty}.
\end{cases}
\end{equation*}
Moreover, if $t> s_{\infty}$, there exists a unique equilibrium measure for $t \phi$.
\end{theorem}

\subsection{Suspension flows}

Let $(\Sigma^+ , \sigma)$ be a topologically mixing countable Markov shift  and $\tau:\Sigma^+ \to \R^+$ a function of summable variations bounded away from zero.
Consider the space
\begin{equation}\label{shift}
Y= \{ (x,t)\in \Sigma^+  \times \R \colon 0 \le t \le\tau(x)\},
\end{equation}
with the points $(x,\tau(x))$ and $(\sigma(x),0)$ identified for
each $x\in \Sigma^+ $. The \emph{suspension semi-flow} over $\sigma$
with \emph{roof function} $\tau$ is the semi-flow $\Phi = (
\varphi_t)_{t \ge 0}$ on $Y$ defined by
\[
 \varphi_t(x,s)= (x,
s+t) \ \text{whenever $s+t\in[0,\tau(x)]$.}
\]
In particular,
\[
 \varphi_{\tau(x)}(x,0)= (\sigma(x),0).
\]

\subsection{Invariant measures} \label{ss:measure}
Let $(Y, \Phi)$ be a suspension semi-flow defined over a countable Markov shift $(\Sigma^+, \sigma)$ with roof function
$\tau:\Sigma^+ \to \R^+$ bounded away from zero. Denote by $\M_{\Phi}$ the space of invariant probability measures for the flow. It follows form a classical result by Ambrose and Kakutani \cite{ak} that every measure $\nu \in \M_{\Phi}$ can be written as
\begin{equation} \label{eq:ak}
\nu= \frac{(\mu \times m)|_{Y}}{(\mu \times m)(Y)},
\end{equation}
where $\mu \in \M_{\sigma}$ and $m$ denotes the one dimensional Lebesgue measure. When $(\Sigma^+, \sigma)$ is a sub-shift of finite type defined on a finite alphabet the relation in equation \eqref{eq:ak} is actually a bijection between
$\M_{\sigma}$ and $\M_{\Phi}$. If $(\Sigma^+, \sigma)$ is a countable Markov shift with roof function bounded away from zero the map defined by
\[\nu \mapsto \frac{(\mu \times m)|_{Y}}{(\mu \times m)(Y)},\]
is surjective. However, it can happen that $(\mu \times m)(Y)= \infty$. In this case the image can be understood as an infinite invariant measure.

The case which is more subtle is when the roof function is only assumed to be positive. We will not be interested in that case here, but we refer to \cite{ijt} for a discussion on the pathologies that might occur.

\subsection{Of flows and semi-flows} \label{s:fs}
In $1972$ Sinai \cite[Section 3]{Sin72} observed that in order to study thermodynamic formalism for suspension  flows it   suffices to study thermodynamic formalism for semi-flows. Denote by $(\Sigma, \sigma)$  a two-sided countable Markov shift. Recall that two continuous functions $\phi, \gamma \in C(\Sigma)$ are said to be \emph{cohomologous} if there exists a bounded continuous function $\psi \in C(\Sigma)$ such that $\phi= \gamma +\psi \circ \sigma -\psi$.  The relevant remark is that thermodynamic formalism for two cohomologous functions is exactly the same. Thus, if every continuous function $\phi  \in C(\Sigma)$ is cohomologous to a continuous function $\gamma  \in C(\Sigma)$ which only depends in future coordinates then thermodynamic formalism for the flow can be studied in the corresponding semi-flow. The next result formalises this discussion.

\begin{proposition} \label{prop:two-sided}
If $\phi \in C(\Sigma)$ has summable variation, then there
exists $\gamma \in C(\Sigma)$ of summable variation cohomologous to $\phi$ such that
$\gamma(x)=\gamma(y)$ whenever $x_i=y_i$ for all $i \geq 0$ (that is,
$\gamma$ depends only on the future coordinates).
\end{proposition}

Proposition \ref{prop:two-sided} has been proved with different regularity assumptions in the compact  setting and in the non-compact case in  \cite[Theorem 7.1]{da}.

\subsection{Abramov and Kac}
The entropy of a flow with respect to an invariant measure can be defined as  the entropy of the corresponding time one map. The following result was proved  by Abramov \cite{a}.

\begin{proposition}[Abramov] \label{prop:as}
Let $\nu \in \M_{\Phi}$ be such that  $\nu=(\mu \times m)|_{Y} /(\mu \times m)(Y)$, where $\mu \in \M_{\sigma}$ then the entropy of $\nu$ with respect to the flow, that we denote $h_\nu(\Phi)$, satisfies
\begin{equation}
h_\nu(\Phi)=\frac{h_\mu(\sigma)}{\int \tau \text{d} \mu}.
\end{equation}
\end{proposition}
%

In Proposition \ref{prop:as} a relation between  the entropy of a measure for the flow and a corresponding measure for the base dynamics was established. We now prove a relation between the integral of a function on the flow with the integral of a related function on the base.  Let  $f \colon Y\to\R$ be a continuous function. Define
$\Delta_f\colon\Sigma^+\to\R$~by
\[
\Delta_f(x):=\int_{0}^{\tau(x)} f(x,t) \, ~{\rm d}t.
\]

 \begin{proposition}[Kac's Lemma] \label{prop:kac}
Let $f\colon Y\to\R$ be a continuous function and $\nu \in \M_{\Phi}$ an invariant measure  that can be written as
$$\nu=\frac{\mu \times m} {(\mu \times m)(Y)},$$
where $\mu \in\M_{\sigma}$, then
\begin{equation*} \label{eq:rela}
\int_{Y}f \, ~{\rm d} \nu= \frac{\int_\Sigma \Delta_f\, ~{\rm d}
\mu}{\int_\Sigma\tau \, ~{\rm d} \mu}.
\end{equation*}
\end{proposition}

Propositions \ref{prop:as} and \ref{prop:kac} together with the relation between the spaces of invariant measures for the flow and for the shift established by Ambrose and Kakutani (see subsection \ref{ss:measure}) allow us to study thermodynamic formalism for the flow by means of the corresponding one on the base.

\subsection{Thermodynamic formalism for suspension flows} \label{s:tf}
Let $(\Sigma^+, \sigma)$ be a topologically mixing countable Markov shift and $\tau:\Sigma^+ \to \R$ a positive function bounded away from zero of summable variations.  Denote by $(Y, \Phi)$ the suspension semi-flow over $(\Sigma^+, \sigma)$ with roof function $\tau$.  Thermodynamic formalism has been studied in this context by several people with different degrees of generality: Savchenko \cite{sav}, Barreira and Iommi \cite{bi}, Kempton \cite{ke} and Jaerisch, Kesseb\"ohmer and  Lamei \cite{jkl}. Thermodynamic formalism for suspension flows where the base $(\Sigma^+, \sigma)$ is a sub-shift of finite type defined on a finite alphabet has been studied, for example, in \cite{br, papo}.
The next result provides equivalent definitions for the pressure, $P_{\Phi}( \cdot)$, on the flow.

\begin{theorem} \label{thm: flow pres}
Let $f:Y \to \R$ be a function such that $\Delta_f:\Sigma^+ \to \R$ is of summable variations. Then the following equalities hold
\begin{eqnarray*}
P_{\Phi}(f)&:=&\lim_{t \to \infty} \frac{1}{t} \log \left(\sum_{\phi_s(x,0)=(x,0), 0<s \leq t} \exp\left( \int_0^s f(\phi_k(x,0)) ~{\rm d}k \right) \chi_{C_{i_0}}(x) \right) \\
&=& \inf\{t \in \R : P_\sigma (\Delta_f - t \tau) \leq 0\} =\sup \{t \in \R : P_\sigma (\Delta_f - t \tau) \geq 0\} \\
&=& \sup \{ P_{\Phi|K}(f) : K\in \cK(\Phi) \},
\end{eqnarray*}
where $\cK(\Phi)$ denotes the space of compact $\Phi-$invariant sets.
\end{theorem}
In particular, the topological entropy of the flow is the unique number $h_{top}(\Phi)$ satisfying
\begin{equation} \label{eq:h}
h_{top}(\Phi) = \inf\{t \in \R : P (- t \tau) \leq 0\}.
 \end{equation}
Note that in this setting the Variational Principle also holds (see \cite{bi, jkl, ke,sav}).

\begin{theorem}[Variational Principle] Let $f:Y \to \R$ be a function such that $\Delta_f:\Sigma^+ \to \R$ is of summable variations. Then
\begin{equation*}
P_\Phi(f)=\sup \left\{ h_{\nu}(\Phi) +\int_Y f ~{\rm d} \nu : \nu\in
\mathcal{M}_{\Phi} \text{ and } -\int_Y f \, ~{\rm d}\nu <\infty \right\}.
\end{equation*}
\end{theorem}
A measure $\nu \in \mathcal{M}_{\Phi}$ is called an \emph{equilibrium measure} for $f$ if
\begin{equation*}
P_\Phi(f)= h_\nu(\Phi) + \int f ~{\rm d} \nu.
\end{equation*}
It was proved in \cite[Theorem 3.5]{ijt} that potentials $f$ for which $\Delta_f$ is locally H\"older have at most one equilibrium measure. Moreover,
the following result (see \cite[Theorem 4]{bi}) characterises functions having equilibrium measures.

\begin{theorem}\label{thm_es}
Let $f \colon Y \to \R$ be a continuous function such
that $\Delta_f$ is of summable variations. Then
there is an equilibrium measure $\nu_f\in \M_\Phi$ for $f$ if  and only if
we have that $P_\sigma(\Delta_f -P_{\Phi}(f) \tau)=0$ and there exists
 an equilibrium measure $\mu_f\in \M_\sigma$ for $\Delta_f -P_{\Phi}(f) \tau$ such that $\int \tau  d\mu_f < \infty$.
\end{theorem}

\begin{remark}
We stress that the situation is more complicated when $\tau$ is not assumed to be bounded away from zero. For results in that setting see \cite{ijt}.
\end{remark}

\section{Entropy and escape of mass}
Over the last few years there has been interest, partially motivated for its connections with number theory, in studying  the relation between entropy and the escape of mass of sequences of invariant measures for diagonal flows on homogenous spaces (see \cite{ekp, elmv, klkm}). Some remarkable results have been obtained bounding the amount of mass that an invariant measure can give to an unbounded part of the domain (a cusp)  in terms of the entropy of the measure (see for example \cite[Theorem A]{ekp} or  \cite[Theorem 1.3]{klkm}). The purpose of this section is to prove similar results in the context of suspension flows defined over countable Markov shifts.   As we will see, the proofs in this setting  suggest a geometrical  interpretation that we pursue in section \ref{e:m}.

Let $(\Sigma^+, \sigma)$ be a topologically mixing countable Markov shift of infinite topological entropy and $\tau: \Sigma^+ \to \R^+$ a potential of summable variations bounded away from zero.  Denote by $(Y, \Phi)$ the associated suspension flow, which we assume to have finite topological entropy. Note that since $(\Sigma^+, \sigma)$ has infinite entropy and $\tau$ is non-negative, the entropy $h_{top}(\Phi)$ of the flow satisfies  $P_\sigma(-h(\Phi) \tau) \leq 0$ (see equation \eqref{eq:h}). Therefore, there exists  a real number $s_{\infty} \in (0, h_{top}(\Phi)]$  such that
\begin{equation*}
P_\sigma(-t\tau)=
\begin{cases}
\textrm{infinite} & \text{ if } t < s_{\infty}; \\

\textrm{finite} & \text{ if } t >s_{\infty}.\end{cases}
\end{equation*}
As it turns out the number $s_{\infty}$ will play a crucial role in our work.

For geodesic flows defined in non-compact manifolds there are vectors that escape through the cusps, they do not exhibit any recurrence property. That phenomena is impossible in the symbolic setting, every point will return to the base after some time.
The following definition describes the set of points that escape on average (compare with an analogous definition given in  \cite{klkm}).

\begin{definition}
We say that a point $(x,t) \in Y$ \emph{escapes on average} if
\begin{equation*}
\lim_{n \to \infty} \frac{1}{n} \sum_{i=0}^{n-1} \tau(\sigma^i x)= \infty.
\end{equation*}
We denote the set of all points which escape on average by $\E_{\A}(\tau)$.
\end{definition}

\begin{remark}
Note that if $\nu \in \M_{\Phi}$ is ergodic and $\nu= (\mu \times m) / (\mu \times m)(Y)$ with $\mu \in \M_{\sigma}$
then Birkhoff's theorem implies that
\begin{equation*}
\lim_{n \to \infty} \frac{1}{n} \sum_{i=0}^{n-1} \tau(\sigma^i x)= \int \tau d \mu.
\end{equation*}
Thus, no measure in $\M_{\Phi}$ is supported on  $\E_{\A}(\tau)$. We can, however, describe the dynamics of the set  $\E_{\A}(\tau)$
studying sequences of measures $\nu_n \in \M_{\Phi}$ such that the associated measures $\mu_n \in \M_{\sigma}$ satisfy
\[\lim_{n \to \infty} \int \tau d \mu_n =\infty.\]
\end{remark}

In our first result we show that a measure of sufficiently large entropy can not give too much weight to the set of points for  which the return time to the base is very high. More precisely,

\begin{theorem} \label{thm:ent_cusp}
Let $c \in (s_{\infty} , h_{top}(\Phi))$.  There exists a constant $C>0$ such that for every  $\nu \in \M_{\Phi}$ with $h_{\nu}(\Phi) \geq c$, we have that
\[\int \tau d \mu \leq  C.\]
\end{theorem}

\begin{proof}
Let $\nu \in \M_{\Phi}$ with  $h_{\nu}(\Phi)=c$ and let $\mu \in \M_{\sigma}$ be the invariant measure satisfying $\nu=(\mu \times m) / ((\mu \times m)(Y))$. By the Abramov  formula we have
\begin{equation*}
h_\mu(\sigma)-c \int \tau d \mu=0.
\end{equation*}
We will consider the straight line $L(t):=h_\mu(\sigma)-t \int \tau d \mu$. Note that $L(c)=0$ and $L(0)=h_\mu(\sigma)$.
Let $s \in (s_{\infty} , c)$. Note that $P_\sigma(-s\tau)< \infty$ and by the variational principle $L(s) \leq P_\sigma(-s\tau)$.
This remark readily implies a bound on the slope of $L(t)$. Indeed,
\begin{equation*}
\int \tau d \mu \leq \frac{P_\sigma(-s \tau)}{c-s}.
\end{equation*}
Thus the constant $C=P_\sigma(-s \tau) / (c-s)$ satisfies the theorem. In order to obtain the best possible constant we have to compute the infimum of the function defined for $s \in  (s_{\infty}, c)$ by
\[s \mapsto \frac{P_\sigma(-s \tau)}{c-s}.\]
\end{proof}

\begin{remark}
We stress that the constant $C$ in Theorem \ref{thm:ent_cusp} depends only on the entropy bound $c$ and not on the measure $\nu$.
\end{remark}

\begin{remark}
An implicit assumption in Theorem \ref{thm:ent_cusp} is that $s_{\infty} < h_{top}(\Phi)$. In Section \ref{sec:sch} we will see that in the geometrical context of geodesic flows this assumption has a very natural  interpretation. Indeed, it will be shown to be equivalent to the \emph{parabolic gap property} (see \cite[Section III]{dalpei} or Definition \ref{def:pgc} for precise definitions).
\end{remark}

\begin{corollary} \label{cor_c}
If $(\Sigma^+, \sigma)$ is a Markov shift defined  countable alphabet satisfying the BIP condition then the best possible constant $C \in \R$  in  Theorem \ref{thm:ent_cusp} is given by
\[C= \frac{P_\sigma(-s_m \tau)}{c-s_m},\]
where $s_m \in \R$ is such that the equilibrium measure $\mu_{s_m}$ for $-s_m \tau$ satisfies
\[ c=  \frac{h_{\mu_{s_m}}(\sigma)}{\int \tau d \mu_{s_m}}.\]
\end{corollary}

\begin{proof}
Since the system satisfies  BIP condition the function $P_\sigma(-s \tau)$, when finite,  is differentiable (see Theorem \ref{bip}).  Moreover, its derivative is given by (see \cite[Theorem 6.5]{sa3}),
\[ \frac{d}{ds} P_\sigma(-s \tau) \Big|_{s=s_m} = -\int \tau d \mu_{s_m},\]
where $\mu_{s_m}$ is the (unique) equilibrium measure for $-s_m \tau$. The critical points of the function $s \mapsto \frac{P_\sigma(-s \tau)}{c-s}$ are those which satisfy
\begin{equation} \label{eq:crit}
(c-s)P'_\sigma(-s \tau) +P_\sigma(-s\tau)=0.
\end{equation}
Equivalently,
\begin{equation*}
-(c-s)\int \tau d\mu_s +h_{\mu_s}(\sigma) - s \int \tau d \mu_s =0.
\end{equation*}
Therefore, equation \eqref{eq:crit} is equivalent to
\begin{equation*}
c= \frac{h_{\mu_s}(\sigma)}{\int \tau d \mu_s}.
\end{equation*}
\end{proof}

In the next Theorem we prove that the entropy of the flow on  $\E_{\A}(\tau)$ is bounded above by $s_{\infty}$ and that, under some additional assumptions, it is actually equal to it. This result could be thought of as a symbolic estimation for the entropy of a flow in a cusp. Theorem \ref{thm:ec} is a refined version of a result first observed in \cite[Lemma 2.5]{fjlr} and used in the context of suspension flows in \cite{ij}.

\begin{theorem} \label{thm:ec}
Let $(\nu_n)_n \in \M_{\Phi}$ be an sequence of invariant probability measure for the flow of the form
\begin{equation*}
\nu_n= \frac{\mu_n \times m}{(\mu_n \times m)(Y)},
\end{equation*}
where $\mu_n \in \M_{\sigma}$. If  $\lim_{n \to \infty} \int \tau d \mu_n=\infty$ then
\[ \limsup _{n \to \infty} h_{\nu_n}(\Phi) \leq s_{\infty}.\]
Moreover, if $s_{\infty}< h_{top}(\Phi)$, then  there exists a sequence $(\nu_n)_n \in \M_{\Phi}$ such that $\lim_{n \to \infty} \int \tau d \mu_n=\infty$ and
 \[ \lim_{n \to \infty} h_{\nu_n}(\Phi) = s_{\infty}.\]
\end{theorem}


\begin{proof}
Observe that the first claim is a direct consequence of Theorem \ref{thm:ent_cusp}. Let us construct now  a sequence  $(\nu_n)_n \in \M_{\Phi}$ with  $\lim_{n \to \infty} \int \tau d \mu_n=\infty$ such that
$\lim_{n \to \infty} h_{\nu_n}(\Phi) = s_{\infty}$. First note that it is a consequence of the approximation property of the pressure, that there exists a sequence of compact invariant sets $(K_N)_N \subset \Sigma$ such that $\lim_{N \to \infty} P_{K_N}(-t \tau)=P_\sigma(-t \tau)$.
In particular, for every $n \in \N$ we have that
\begin{equation}  \label{eq:infy}
\lim_{N \to \infty} P_{K_N}\left(-\left(s_{\infty} -1/n\right)\tau\right)= \infty.
\end{equation}
For the same reason, for any $n \in \N$ and $N \in \N$ we have that
\begin{equation}
P_{K_N}\left(-\left(s_{\infty} +1/n\right)\tau\right) \leq P_\sigma\left(-\left(s_{\infty} +1/n\right)\tau\right) <\infty.
\end{equation}
Thus, given $n \in \N$ there exists $N \in \N$ such that
\begin{equation*}
n^2 < \frac{P_{K_N}\left(-\left(s_{\infty} -1/n\right)\tau\right)- P_{K_N}\left(-\left(s_{\infty} +1/n\right)\tau\right)}{2/n}.
\end{equation*}
Since the function $t \mapsto P_{K_N}\left(-t\tau\right)$ is real analytic, by the mean value theorem there exists
$t_n \in [s_{\infty} -1/n, s_{\infty} +1/n]$, such that $ P'_{K_N}\left(-t_n\tau\right)>n^2$. Denote by $\mu_n$ the equilibrium measure for $-t_n\tau$ in $K_N$. We have that
\[n^2 <P'_{K_N}\left(-t_n\tau\right) = \int \tau d \mu_n.\]
In particular the sequence $(\mu_n)_n$ satisfies
\begin{equation*}
\lim_{n \to \infty} \int \tau d \mu_n = \infty.
\end{equation*}
Since $s_{\infty} < h_{top}(\Phi)$ we have that for $n \in \N$ large enough
\begin{equation*}
h_{\mu_n}(\sigma) - t_n \int \tau d \mu_n >0.
\end{equation*}
In particular
\begin{equation*}
t_n < \frac{h_{\mu_n}(\sigma)}{\int \tau d \mu_n}.
\end{equation*}
Since $t_n \in (s_{\infty}- 1/n ,s_{\infty}+ 1/n)$ we have that
\begin{equation}
s_{\infty} = \lim_{n \to \infty} t_n \leq \lim_{n \to \infty}  \frac{h_{\mu_n}(\sigma)}{\int \tau d \mu_n} = \lim_{n \to \infty } h_{\nu_n}(\Phi).
\end{equation}
But we already proved that the limit can not be larger than $s_{\infty}$, thus the result follows.
\end{proof}

\begin{corollary}\label{cor:ent_inf}
If $s_\infty < h_{top}(\Phi)$, then the entropy at infinity $h_\infty(Y,\Phi)$ of the suspension flow satisfies
$$h_\infty(Y,\Phi) = s_\infty.$$
\end{corollary}
\begin{proof}
Let $(\nu_n)$ be a sequence of $\Phi$-invariant probability measures such that $\nu_n \rightharpoonup 0$. Since
$$\int f d\nu_n = \frac{\int \Delta_f d\mu_n}{\int \tau d\mu_n},$$
the only way to have $\int f d\nu_n \rightarrow 0$ for every continuous function $f\in C^0(Y)$ with compact support, is that
$$\lim_{n\to\infty}\int \tau d\mu_n = \infty.$$
Hence, Theorem \ref{thm:ec} implies that $\limsup_{n\to \infty} h_{\nu_n}(\Phi) \leq s_{\infty}$. On the other hand, Theorem \ref{thm:ec} also says that there exists a sequence $(\nu_n)$ of $\Phi$-invariant probability measures such that $\lim_{n\to\infty}\int \tau d\mu_n = \infty$ and
$$\lim_{n\to \infty} h_{\nu_n}(\Phi) = s_{\infty}.$$
Again, since $\int f d\nu_n = \frac{\int \Delta_f d\mu_n}{\int \tau d\mu_n}$, we have that $\nu_n \rightharpoonup 0$, and the conclusion follows.
\end{proof}

\section{The geodesic flow on Extended Schottky Groups} \label{sec:sch}

\subsection{Some preliminaries in negative curvature}\label{sub_poi}

Let $X$ be a Hadamard manifold with pinched negative sectional curvature, that is a complete simply connected Riemannian manifold whose sectional curvature $K$ satisfies  $-b^{2}\leq K \leq -1$ (for some fixed $b\geq 1$). Denote by $\partial X$ the boundary at infinity of $X$. Finally, denote by $d$ the Riemannian distance on $X$. A crucial object on the study of the dynamics of the geodesic flow is the \emph{Busemann function}. Let $\xi\in\partial X$ and $x,y\in X$. For every geodesic ray $t\mapsto \xi_t$ pointing to $\xi$, the limit
$$B_{\xi}(x,y):=\lim_{t\to\infty}[d(x,\xi_t)-d(y,\xi_t)],$$
always exists, and is independent of the geodesic ray $\xi_{t}$ since $X$ has negative sectional curvature. The Busemann function $B:\partial X\times X^{2}\rightarrow \R$ is the continuous function defined as $B(\xi,x,y)\mapsto B_{\xi}(x,y)$. An (open) \emph{horoball} based in $\xi$ and passing through $x$ is the set of $y\in X$ such that $B_{\xi}(x,y)>0$. In the hyperbolic case, when $X=\disk$, an open horoball based in $\xi\in\mathbb{S}^{1}$ and passing trough $x\in\disk$ is the interior of an euclidean circle containing $x$ and tangent to $\mathbb{S}^{1}$ at $\xi$.

Recall that every isometry of $X$ can be extended to a homeomorphism of $X\cup\partial X$. A very important property of the Busemann function is the following. If $\phi:X\to X$ is an isometry of $X$, then for every $x,y\in X$, we have
\begin{equation}\label{prop:busemann_invariant}
B_{\phi\xi}(\phi x,\phi y)=B_{\xi}(x,y).
\end{equation}

Let $o\in X$ be a reference point, which is often called the \emph{origin} of $X$. The unit tangent bundle $T^1 X$ of $X$ can be identified with $\partial^2 X \times \R$, where $\partial^2 X=(\partial X\times\partial X)\setminus\text{diagonal}$, via Hopf's coordinates. A vector $v\in T^{1}X$ is identified with $(v^{-},v^{+},B_{v^{+}}(o,\pi(v)))$, where $v^{-}$ (resp. $v^{+}$) is the negative (resp. positive) endpoint of the geodesic defined by $v$. Here $\pi:T^{1}X\rightarrow X$ is the natural projection of a vector to its base point. Observe that the geodesic flow $(g_t):T^1 X\to T^1 X$ acts by translation in the third coordinate of this identification.\\

Another crucial object in this setting is the Poincar\'e series. It is intimately related to the topological entropy of the geodesic flow.

\begin{definition} \label{def:ps}
Let $\Gamma$ be a discrete subgroup of isometries of $X$ and let $x\in X$. The \emph{Poincar\'e series} $P_{\Gamma}(s,x)$ associated with $\Gamma$ is defined by
\begin{equation*}
P_{\Gamma}(s,x):=\sum_{g\in \Gamma}e^{-sd(x,gx)}.
\end{equation*}
The \emph{critical exponent} $\delta_\Gamma$ of $\Gamma$ is
\begin{equation*}
\delta_{\Gamma}:= \inf \left\{ s \in \R : 	P_{\Gamma}(s,x) < \infty		\right\}.
\end{equation*}
The group $\Gamma$ is said to be of \emph{divergence type} (resp. \emph{convergence type}) if $P_{\Gamma}(\delta_{\Gamma},x)=\infty$ (resp. $P_{\Gamma}(\delta_{\Gamma},x)<\infty$).
\end{definition}

\begin{remark}
As the sectional curvature of $X$ is bounded from below, the critical exponent is finite. Moreover, by the triangle inequality, it is independent of $x\in X$.
\end{remark}

The isometries of $X$ are categorized in three types. Those fixing an unique point in $X$ called \emph{elliptic isometries}. Those fixing an unique point in $\partial X$ called \emph{parabolic isometries}. And finally, those fixing uniquely two points in $\partial X$ called \emph{hyperbolic isometries}. For $g$ a non-elliptic isometry of $X$, denote by $\delta_{g}$ the critical exponent of the group $<g>$. If $g$ is hyperbolic it is fairly straightforward to see that $\delta_g=0$ and that the group $<g>$ is of divergence type. If $g$ is parabolic, it was shown in \cite[Theorem III.1]{dalpei}, that $\delta_g \geq \frac{1}{2}$.


Let $\Gamma$ be a discrete subgroup of isometries of $X$. Denote by $\Lambda$ the limit set of $\Gamma$, that is, the set $\Lambda=\overline{\Gamma\cdot o}\setminus \Gamma\cdot o$. The group $\Gamma$ is \emph{non-elementary} if $\Lambda$ contains infinitely many elements. We recall the following fact proved in \cite[Proposition 2]{dalotpei}.

\begin{theorem}\label{thm.DP.1} Let $\Gamma$ be a non-elementary discrete subgroup of isometries of a Hadamard  manifold $X$. If $G$ is a divergence type subgroup of $\Gamma$ and its limit set is strictly contained in the limit set of $\Gamma$, then $\delta_\Gamma>\delta_G$.\end{theorem}

In particular, if $\Gamma$ is a non-elementary discrete group of isometries and there is an element $g\in \Gamma$ such that $<g>$ is of divergence type, then $\delta_{\Gamma}>\delta_{<g>}$ (see also \cite[Theorem III.1]{dalpei}). Note that a non-elementary group always contains a hyperbolic isometry (in fact, infinitely many non-conjugate of them), hence a non-elementary group $\Gamma$ always satisfies $\delta_\Gamma>0$.\\

We end this subsection  giving an important relation between the topological entropy of the geodesic flow and the critical exponent of a group. Let $X$ be a Hadamard manifold with pinched negative sectional curvature and let $\Gamma$ be a non-elementary free-torsion discrete subgroup of isometries of $X$. Denote by $(g_t):T^1X/\Gamma \to T^1 X/ \Gamma$ the geodesic flow on the unit tangent bundle of the quotient manifold $X/\Gamma$. Otal and Peign\'e \cite[Theorem 1]{op} proved that, if the derivatives of the sectional curvature are uniformly bounded, then the topological entropy $h_{top}(g)$ of the geodesic flow equals the critical exponent of the Poincar\'e series of the group $\Gamma$, that is
\begin{equation}\label{eq:OP}
h_{top}(g)= \delta_\Gamma.
\end{equation}
We stress the fact that the assumption on the derivatives of the sectional curvature is crucial in order to compute the topological entropy of the geodesic flow. This assumption implies the H\"older regularity of the strong unstable and stable foliations (see for instance \cite[Theorem 7.3]{pps}), which is used in the proof of \cite[Theorem 1]{op}.

\subsection{The symbolic model for extended Schottky groups}\label{ss:sr}

In this section we recall the definition of an extended Schottky group. To the best of our knowledge this definition has been introduced in 1998 by Dal'bo and Peign\'e \cite{dalpei}. The basic idea is to extend the classical notion of Schottky groups to the context of manifolds where the non-wandering set of the geodesic flow is non-compact.

Let  $X$ be a Hadamard manifold as in subsection \ref{sub_poi}. Let $N_{1} , N_{2} $ be non-negative integers such that $N_1+N_{2}\geq 2$ and $N_2\ge 1$. Consider $N_{1}$ hyperbolic isometries $h_{1},...,h_{N_{1}}$ and $N_{2}$ parabolic ones $p_{1},...,p_{N_{2}}$ satisfying the following conditions:

\begin{enumerate}

\item[(C1)] For $1\leq i\leq N_{1}$ there exist in $\partial X$ a compact neighbourhood $C_{h_{i}}$ of the attracting point $\xi_{h_{i}}$ of $h_{i}$ and a compact neighbourhood $C_{h_{i}^{-1}}$ of the repelling point $\xi_{h_{i}^{+}}$ of $h_{i}$, such that
    $$h_{i}(\partial X\setminus C_{h_{i}^{-1}})\subset C_{h_{i}}.$$

\item[(C2)] For $1\leq i\leq N_{2}$ there exists in $\partial X$ a compact neighbourhood $C_{p_{i}}$ of the unique fixed point $\xi_{p_{i}}$ of $p_{i}$, such that
$$\forall n\in\Z^{\ast} \quad p_{i}^{n}(\partial X\setminus C_{p_{i}})\subset C_{p_{i}}.$$

\item[(C3)] The $2N_{1}+N_{2}$ neighbourhoods introduced in $(1)$ and $(2)$ are pairwise disjoint.

\item[(C4)] The elementary parabolic groups $<p_{i}>$, for $1\leq i\leq N_{2}$, are of divergence type.

\end{enumerate}

The group $\Gamma=<h_{1},...,h_{N_{1}},p_{1},...,p_{N_{2}}>$ is a non-elementary free group which acts properly discontinuously and freely on $X$ (see \cite[Corollary II.2]{dalpei}). Such a group $\Gamma$ is called an \emph{extended Schottky group}. Note that if $N_2=0$,  that is the group $\Gamma$ only contains hyperbolic elements, then $\Gamma$ is a classical Schottky group and its geometric and dynamical properties are well understood. Indeed, in that case, the non-wandering set $\Omega\subset T^1X/\Gamma$ of the geodesic flow is compact, which implies that $(g_t)|_{\Omega}$ is an Axiom A flow. If $N_2\geq 1$, then $X/\Gamma$ is a non-compact manifold and $\Omega$ is a non-compact subset of $T^1X/\Gamma$. Figure 1 below is an example of a Schottky group acting on the hyperbolic disk $\disk$. It has two generators, one hyperbolic and the other parabolic.

\begin{center}
\begin{tikzpicture}[scale=2.5]

\draw (0,0) circle (1);
\draw (0,1) arc (180:298:0.6);
\draw (0,1) arc (0:-118:0.6);

\draw (0,0) node[above]{$o$};

\draw (0.3333,-0.9428) arc (200:72:0.5);

\draw (-0.3333,-0.9428) arc (-20:108:0.5);

\draw (0,1) node[above]{$C_{p}$} node{$\bullet$};

\draw [very thick] (0,1) arc (90:28:1) (0,1) arc (90:152:1);

\draw [thick] (0,0.4) arc (270:290:0.4) [->];
\draw (0,0.4) node[below]{$p$};
\draw [thick] (0,0.4) arc (270:250:0.4);

\draw [thick] (0,-0.5) arc (90:70:0.5) [->];
\draw (0,-0.5) node[above]{$h$};
\draw [thick] (0,-0.5) arc (90:110:0.5);

\draw (0,-1.3) node{Figure 1. Schottky Group $\Gamma=<h,p>$.};

\draw (0,0) node{$\bullet$};

\draw [very thick] (-0.3333,-0.9428) arc (250:197:1) ;

\draw (-0.715,-0.69) node[below left]{$C_{h^{-1}}$} node{$\bullet$};

\draw [very thick] (0.3333,-0.9428) arc (290:343:1) ;

\draw (0.715,-0.69) node{$\bullet$} node[below right]{$C_{h}$};

\end{tikzpicture}
\end{center}

Let $\mathcal{A}^{\pm}=\{h_{1}^{\pm 1},...,h_{N_{1}}^{\pm 1},p_{1},...,p_{N_{2}}\}$. We now define an extra hypothesis that we will use in large parts of this paper.

\begin{enumerate}
\item[(C5)] Let $a_{1},a_{2}\in\mathcal{A}^{\pm}$ be such that $a_{1}\neq a_{2}^{\pm 1}$. Then, there exists a point $o\in X$ such that for every $\xi\in C_{a_{1}}$, we have
$$B_{\xi}(a^{\pm 1}_{2}o,o)>0.$$
\end{enumerate}
In other words, we will assume that every horoball based in $C_{a_{1}}$ and passing through $a^{\pm 1}_{2}o$ contains the origin in his interior (see for instance Figure 2). This condition is not very restrictive as we will see in Proposition \ref{SchottkyCondition5}.

\begin{center}
\begin{tikzpicture}[scale=2.5]

\draw (0,0) circle (1);
\draw (0,1) arc (180:298:0.6);
\draw (0,1) arc (0:-118:0.6);

\draw (0.3333,-0.9428) arc (200:72:0.5);
\draw (-0.3333,-0.9428) arc (-20:108:0.5);

\draw (0,0) node{$\bullet$} node[above]{$o$};

\draw (-0.6,-0.55) node{$\bullet$} node[above]{$z$};
\draw (0.3,0.9539) node{$\bullet$} node[above]{$\xi$};
\draw [dashed] (0.0296,0.0941) circle (0.9009);
\draw [dashed] (0.15,0.4769) circle (0.4999);

\draw (-0.24,-0.776) node{$\bullet$};
\draw [densely dotted] (-0.24,-0.776)--(0.3,0.9539);
\draw (-0.14,-0.826)--(0.1,-0.05) [>=stealth, <->] node[fill=white,midway,scale=0.7]{$B_{\xi}(z,o)$};

\draw (0,-1.3) node{Figure 2.};

\end{tikzpicture}
\end{center}

\noindent
For $a\in\mathcal{A}^{\pm}$ denote by $U_a$ the convex hull in $X\cup \partial X$ of the set $C_{a}$.
\begin{lemma}\label{lem:bus_dist_comparison} Let $X$ be a Hadamard manifold with pinched negative sectional curvature and let $\Gamma$ be an extended Schottky group. Fix $o\in X$. Then, there exists an universal constant $C>0$ (depending only on the generators of $\Gamma$ and the fixed point $o$) such that for every $a_1,a_2\in\mathcal{A}^{\pm}$ satisfying $a_1\neq a_2^{\pm 1}$, and for every $x\in U_{a_1}$ and $y\in U_{a_2}$, we have
\begin{equation}\label{eq:lem:bus_dist}
d(x,y)\geq d(x,o)+d(y,o)-C.
\end{equation}
\end{lemma}
\begin{proof}
Since $C_{a_1}$ and $C_{a_2}$ are disjoint, for every $a_1,a_2\in\mathcal{A}^{\pm}$ satisfying $a_1\neq a_2^{\pm 1}$, the same happens for the sets $U_{a_1}$ and $U_{a_2}$. Let $x\in U_{a_1}$ and $y\in U_{a_2}$. The geodesic segments $[o,x]$ and $[o,y]$ form an angle uniformly bounded below, hence $d(x,y)\geq d(x,o)+d(y,o)-C$ for an universal constant $C>0$.
\end{proof}

\begin{proposition}\label{SchottkyCondition5} Let $X$ be a Hadamard manifold with pinched negative sectional curvature and $\Gamma=<h_{1},...,h_{N_{1}},p_{1},...,p_{N_{2}}>$ an extended Schottky group. Then, for every $o\in X$ there exists an integer $N\geq 1$ such that the group defined by $<h^{N}_{1},...,h^{N}_{N_{1}},p^{N}_{1},...,p^{N}_{N_{2}}>$ satisfies the Condition (C5).
\end{proposition}
\begin{proof} Let $a_{1},a_{2}\in\mathcal{A}^{\pm}$ and $\xi\in C_{a_{1}}$. Denote by $U_{a^{n}_{2}}$ the convex hull in $X\cup \partial X$ of the set $C_{a^{n}_{2}}$, for $n\geq 1$. Let $z_n\in U_{a^{n}_{2}}$ be such that $B_\xi(z_n,o)$ attains its minimum. Since $B$ is a continuous function and $C_{a_{1}}$ is a compact set, it is sufficient to prove that $B_{\xi}(z_n,o)>0$ for all $n$ large enough. Consider $(\xi_t)=[o,\xi_t)$ the geodesic ray starting in $o$ and pointing to $\xi$. Observe that there exists $T>0$ such that for every $t \geq T$, we have that $\xi_t$ belongs to $U_{a_{1}}$. Since $C_{a_{1}}$ and $C_{a_{2}}$ are disjoint, by Lemma \ref{lem:bus_dist_comparison} there exists a universal constant $C>0$ such that $d(z_n,\xi_{t})\geq d(\xi_{t},o)+d(z_n,o)-C$, for every $t\geq T$. Therefore,
\begin{eqnarray*}
B_{\xi}(z_n,o)&=&\lim_{t\to+\infty} d(z_n,\xi_{t})-d(\xi_{t},o)\\
&\geq& \lim_{t\to+\infty} d(\xi_{t},o)+d(z_n,o)-C-d(\xi_{t},o)\\
&=& d(z_n,o)-C.
\end{eqnarray*}
Since $d(z_n,o)\to \infty$ as $n\to \infty$, there exists a $N\geq 1$ such that $d(z_n,o)>C$ for every $n\geq N$. In particular, we have $B_\xi(z_n,o)>0$ for every $n\geq N$ and the conclusion follows.
\end{proof}

In \cite{dalpei} the authors proved that there exists a $(g_{t})$-invariant subset $\Omega_0$ of $T^{1}X/\Gamma$, contained in the non-wandering set of $(g_{t})$, such that $(g_{t})|_{\Omega_0}$ is topologically conjugated to a suspension flow over a countable Markov shift $(\Sigma,\sigma)$. The Theorem below summarizes their construction together with some dynamical properties.

\begin{theorem}\label{thm:geod_flow_susp_flow} Let $X$ be a Hadamard manifold with pinched negative sectional curvature and let $\Gamma$ be an extended Schottky group. Suppose that $\Gamma$ satisfies Condition (C5). Then, there exists a $(g_{t})$-invariant subset $\Omega_0$ of $T^{1}X/\Gamma$, a countable Markov shift $(\Sigma,\sigma)$ and a function $\tau:\Sigma\to\R$, such that
\begin{enumerate}
\item[(1)] the function $\tau$ is locally H\"older and is bounded away from zero,
\item[(2)] the geodesic flow $(g_{t})|_{\Omega_0}$ over $\Omega_0$ is topologically conjugated to the suspension flow over $\Sigma$ with roof function $\tau$,
\item[(3)] the Markov shift $(\Sigma,\sigma)$ satisfies the BIP condition,
\item[(4)] if $N_1+N_2\geq 3$, then $(\Sigma,\sigma)$ is topologically mixing.
\end{enumerate}
\end{theorem}
\begin{proof}
Let $\mathcal{A}=\{h_{1},...,h_{N_{1}},p_{1},...,p_{N_{2}}\}$ and consider the symbolic space $\Sigma$ defined by
$$\Sigma=\{(a_{i}^{m_{i}})_{i\in \Z}: a_{i}\in \mathcal{A}, m_{i}\in\Z \mbox{ and } a_{i+1}\neq a_{i} \forall i\in\Z\}.$$
Note that the space $\Sigma$ is a sequence space defined on the countable alphabet $\{ a_i^{m} : a_i \in \mathcal{A} , m \in \Z\}$. Let $\Lambda^{0}$ be the limit set $\Lambda$ minus the $\Gamma$-orbit of the fixed points of the elements of $\mathcal{A}$. We denote by $\tilde{\Omega}_0$ the set of vectors in $T^{1}X$ identified with $(\Lambda^{0}\times\Lambda^{0}\setminus diagonal)\times \R$ via Hopf's coordinates. Finally, define $\Omega_0:=\tilde{\Omega}_0/\Gamma$, where the action of $\Gamma$ is given by $$\gamma\cdot (\xi^-,\xi^+, s)=(\gamma(\xi^-),\gamma(\xi^+), s-B_{\xi^+}(o,\gamma^{-1}0)).$$
Observe that $\Omega_0$ is invariant by the geodesic flow.

Fix now $\xi_{0}\in \partial X\setminus \bigcup_{a\in\mathcal{A}}C_{a^{\pm}}$, where $C_{a^{\pm}}=C_{a}\cup C_{a^{-1}}$.  Dal'bo and Peign\'e \cite[Property II.5]{dalpei} established the following coding property: for every $\xi\in\Lambda^{0}$ there exists an unique sequence $\omega(\xi)=(a_{i}^{m_{i}})_{i\geq 1}$ with $a_{i}\in\mathcal{A}$, $m_{i}\in\Z^{\ast}$ and $a_{i+1}\neq a_{i}$ such that
$$\lim_{k\rightarrow\infty} a_{1}^{m_{1}}...a_{k}^{m_{k}}\xi_{0}=\xi.$$
For each $a\in\mathcal{A}$ define $\Lambda^{0}_{a^{\pm}}=\Lambda^{0}\cap C_{a^{\pm}}$ and set $\partial^{2}\Lambda^{0}=\bigcup_{\stackrel{\alpha,\beta\in\mathcal{A}}{\alpha\neq\beta}}\Lambda^0_{\alpha^\pm}\times \Lambda^0_{\beta^\pm}$. For any pair $(\xi^{-},\xi^{+})\in \partial^{2}\Lambda^{0}$, if $a^{m}$ is the first term of the sequence $\omega(\xi^{+})$, define $\tilde{\tau}(\xi^{+})=B_{\xi^{+}}(o,a^{m}o)$ and $\overline{T}(\xi^{-},\xi^{+})=(a^{-m}\xi^{-},a^{-m}\xi^{+})$. Define $\overline{T}_{\tau}$ by the formula
$$\overline{T}_{\tau}(\xi^{-},\xi^{+},s)=(\overline{T}(\xi^{-},\xi^{+}),s-\tilde{\tau}(\xi^{+})).$$
Observe that $\overline{T}_{\tau}$ maps $\partial^{2}\Lambda^{0}\times \R$ to itself. The set $\Omega_0$ can be identified with the quotient $\partial^{2}\Lambda^{0}\times \R/<\overline{T}_{\tau}>$. In order to prove this claim,  we first proceed to explain how $\overline{T}_{\tau}$ gives a suspension flow.

Let $(\xi^{-},\xi^{+})\in \partial^{2}\Lambda^{0}$. Suppose that $\omega(\xi^+)=(a_{i}^{m_{i}})$ and let $\gamma_n$ be the element of $\Gamma$ defined as $\gamma_n=a_{1}^{m_{1}}...a_{n}^{m_{n}}$ for $n\geq 1$ and $\gamma_0=Id$. The geodesic determined by $(\xi^{-},\xi^{+})$ in $X$ intersects the horosphere based in $\xi^{+}$ and passing through $\gamma_n o$ in only one point $x^n_{o,\xi^{-},\xi^{+}}$. Denote by $\tilde{v}^n_{o,\xi^{-},\xi^{+}}$ the vector in $T^{1}X$ based in $x^n_{o,\xi^{-},\xi^{+}}$ and pointing to $\xi^{+}$ (see Figure 3 below). Finally, denote by $v^n_{o,\xi^{-},\xi^{+}}$ the projection of $\tilde{v}^n_{o,\xi^{-},\xi^{+}}$ over $T^{1}X/\Gamma$. Set
$$S=\{v^0_{o,\xi^{-},\xi^{+}}:(\xi^{-},\xi^{+})\in \partial^{2}\Lambda^{0}\}\subset T^1 X/\Gamma.$$
Condition (C5) implies that the first return time of every $v^0_{o,\xi^{-},\xi^{+}}\in S$ is given by $\tilde{\tau}(\xi^{+})$ (which is the distance between the horospheres based in $\xi^{+}$ and passing through $o$ and $\gamma_1 o$). Moreover, $g_{\tilde{\tau}(\xi^{+})}(v^0_{o,\xi^{-},\xi^{+}})=v^{1}_{o,\xi^{-},\xi^{+}}$ (see Figure 3).

\begin{center}
\begin{tikzpicture}[scale=2.7]

\draw (0,0) circle (1);
\draw[densely dotted] (0,1) arc (180:298:0.6);
\draw[densely dotted] (0,1) arc (0:-118:0.6);

\draw[densely dotted] (0,1) arc (180:340:0.18);

\draw[densely dotted] (0.38,0.915) arc (150:332:0.15);
\draw[densely dotted] (0.66,0.75) arc (135:305:0.15);


\draw (0.6,0.8) arc (140:233:1) node[below right]{$\xi^{-}$};
\draw[color=blue,line width=1pt] (0.4,-0.1) arc (195:155:1);
\draw (0.47,0.6) node{$\bullet$} node[scale=0.6,left]{$v^1_{o,\xi^{-},\xi^{+}}$};
\draw[line width=1.3pt][>=stealth,->] (0.47,0.6)--(0.54,0.75) ;

\draw[dashed] (0.3,0.4) circle (0.5);
\draw (0.4,-0.1) node{$\bullet$};
\draw (0.4,-0.06) node[scale=0.6,below right]{$v^0_{o,\xi^{-},\xi^{+}}$};
\draw[line width=1.3pt][>=stealth,->] (0.4,-0.1)--(0.37,0.1) ;

\draw (0.6,0.8) node[above right]{$\xi^{+}$};

\draw [color=blue] (0.35,0.27) node[scale=0.6,left]{$\tilde{\tau}(\xi^+)$};

\draw[densely dotted] (0.3333,-0.9428) arc (200:72:0.5);
\draw[densely dotted] (-0.3333,-0.9428) arc (-20:108:0.5);

\draw (0,0) node{$\bullet$} node[above]{$o$};

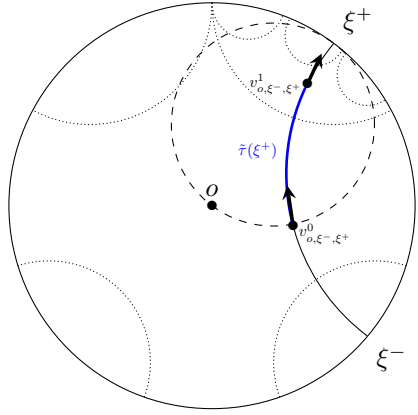
\draw (0,-1.3) node{Figure 3. Cross section for $\Gamma=<h,p>$};

\end{tikzpicture}
\end{center}

More generally, for every $n\geq 0$ the first return time of $v^{n}_{o,\xi^{-},\xi^{+}}\in S$ is given by $\tilde{\tau}(\gamma^{-1}_n\xi^{+})$ and
$$g_{\tilde{\tau}(\gamma^{-1}_n\xi^{+})}(v^{n}_{o,\xi^{-},\xi^{+}})=v^{n+1}_{o,\xi^{-},\xi^{+}}.$$
Note that for every vector $v\in\Omega_0$ there exists a time $t\in\R$ such that $g_{t}v$ belongs to $S$ (otherwise, a lift of $v$ to $T^{1}M$ would have its positive endpoint $v^+$ in the $\Gamma$-orbit of a fixed point of an element in $\mathcal{A}$). This fact give us the identification.

The coding property implies that the set $\partial^{2}\Lambda^{0}$ is identified with $\Sigma$ by considering $(\xi^{-},\xi^{+})$ as a bilateral sequence $(\omega^*(\xi^{-}),\omega(\xi^{+}))$. If $\omega(\xi^-)=(b_i^{n_i})_{i\ge 1}$, we define $\omega^*(\xi^-)$ as the sequence $(...,b_2^{-n_2},b_1^{-n_1})$, then $(\omega^*(\xi^{-}),\omega(\xi^{+}))$ represent the concatenated sequence. Let $\Sigma^+$ be the one sided symbolic space obtained from $\Sigma$ by forgetting the negative time coordinates. We define the function $\tau:\Sigma^+\to\R$ as
$$\tau(x)=\tau(\omega^{-1}(x))=B_{\omega^{-1}(x)}(o,a^{m}o),$$
where $w:\Lambda^0\to\Sigma$ is the coding function and $a^m$ the first symbol in $w^{-1}(x)$. We extend $\tau$ to $\Sigma$ by making it independent of the negative time coordinates. We therefore have, that the geodesic flow $(g_t)$ on $\Omega_0$ can be coded as the suspension flow on $Y=\{(x,t)\in \Sigma\times \R : 0\leq t\leq \tau(x)\}/\sim$. This implies (2) in the conclusion of Theorem \ref{thm:geod_flow_susp_flow}. Property (1) follows from Lemma \ref{lh} below.

\begin{lemma} \label{lh} Under the hypothesis of Theorem \ref{thm:geod_flow_susp_flow}, the function $\tau:\Sigma \to \R$  depends only on future coordinates, it is locally H\"older and it is bounded away from zero.
\end{lemma}

\begin{proof}
The fact that it depends on the future is by definition and the regularity property was established in \cite[Lemma VII.]{dalpei}. Let $x\in\Sigma^+$ and $\xi=\omega^{-1}(x)$ the point in $\Lambda^{0}$ associated to $x$ by the coding property. The function $\tau$ satisfies
$$\tau(x)=B_{\xi}(o,a^{m}o)=B_{a^{-m}\xi}(a^{-m}o,o).$$
The last term above is greater than 0 by Condition (C5). Indeed, note that $a^{-m}\xi=\omega^{-1}(\sigma x)\not\in C_{a^{\pm}}$ and $a^{-m}o$ is contained in the convex-hull of $C_{a^{\pm}}$ in $X\times \partial X$, hence Condition (C5) applies directly. Since the domains $C_{a}$, for $a\in\mathcal{A}^{\pm}$, are compact and $B$ is continuous, there exists an uniform strictly positive lower bound for $B_{a^{-m}\xi}(a^{-m}o,o)$. Thus $\tau$ is bounded away from zero.
\end{proof}

\begin{lemma} \label{lema:mix} Under the hypothesis of Theorem \ref{thm:geod_flow_susp_flow} the countable Markov shift $(\Sigma,\sigma)$ satisfies the BIP condition. Moreover, if $N_1+N_2\geq 3$, then the countable Markov shift $(\Sigma,\sigma)$ is topologically mixing.
\end{lemma}

\begin{proof}
It is not hard to see that the set $\mathcal{A}$ satisfies the required conditions in order for $(\Sigma, \sigma)$ to be BIP (see definition \ref{def:bip}). Suppose now that $N_1+N_2\geq 3$. Recall that the Markov shift  $(\Sigma,\sigma)$  is topologically mixing if for every $a, b \in \{ a_i^{m} : a_i \in \mathcal{A} , m \in \Z\}$ there exists $N(a,b) \in \N$ such that for every $n >N(a,b)$ there exists an admissible word of length $n$ of the form $a i_1 i_2 \dots i_{n-1}b$. The set of allowable sequences is given by
\begin{equation*}
\{(a_{i}^{m_{i}})_{i\in \Z}: a_{i}\in \mathcal{A}, m_{i}\in\Z \mbox{ and } a_{i+1}\neq a_{i} \forall i\in\N\}.
\end{equation*}
Since $N_1 + N_2 \geq 3$ then given any pair of symbols in $\{ a_i^{m} : a_i \in \mathcal{A} , m \in \Z\}$, say $a_1^{m_1}, a_{2}^{m_2}$ we can consider the symbol $a_3 \notin \{a_1, a_2\}$. Hence the following words are admissible:
\[ a_1^{m_1} a_{3}a_{1}a_{3} \dots a_{1}		a_2^{m_2}  \text{ , }	a_1^{m_1} a_{3}a_{1}a_{3}a_1 \dots a_{3} a_2^{m_2}.	\]
Thus,  the system is topologically  mixing.
\end{proof}
Since Lemma \ref{lema:mix} above shows the points (3) and (4), we have concluded the proof of Theorem \ref{thm:geod_flow_susp_flow}.

\end{proof}

\begin{remark} Under the conditions $(C5)$ and $N_1+N_2\geq 3$ we have proved that $(\Sigma,\sigma)$ is a topologically mixing countable Markov shift satisfying the BIP condition (Lemma \ref{lema:mix}) and that the roof function $\tau$ is locally H\"older and bounded away from zero (Lemma \ref{lh}). Therefore, the associated suspension semi-flow $(Y, \Phi)$ can be studied with the techniques presented in Section \ref{s:pre}.
\end{remark}

So far we have proved that the geodesic flow restricted to the set $\Omega_0$ can be coded by a suspension flow over a countable Markov shift. We now describe, from the ergodic point of view, the geodesic flow in the complement $\left( T^1X/\Gamma \right) \setminus \Omega_0$ of $\Omega_0$. We denote  by $\M_{\Omega_0}$ the space of $(g_t)$-invariant probability measures supported in the set where we have coding, in other words in $\partial^2\Lambda_0\times \R/<\overline{T}_{\tau}>$. We describe the difference between the space $\M_{\Omega_0}$ and the space $\M_g$ of all $(g_t)$-invariant probability measures. Recall that in $\Gamma$ there are hyperbolic isometries $h_1,...,h_{N_1}$, each of which fixes a pair of points in $\partial X$. The geodesic connecting the fixed points of $h_i$ will descend to a closed geodesic in the quotient by $\Gamma$. We denote $\nu^{h_i}$ the probability measure equidistributed along such a geodesic.

\begin{proposition} \label{prop:me}  The set of ergodic measures in $\M_g\setminus \M_{\Omega_0}$ is finite, those are exactly the set $\{\nu^{h_i}: 1\leq i \leq N_1\}$. Moreover, for every $\nu\in \M_g\setminus \M_{\Omega_0}$ we have $h_\nu(g)=0$.
\end{proposition}

\begin{proof}
Let $\nu\in \M_g\setminus \M_{\Omega_0}$ be an ergodic measure. Take $v$ a generic vector for $\nu$. Since a generic vector is recurrent, the orbit $g_t v$ does not goes to infinity, therefore $v^+$ is not parabolic. Now consider the case when $v$ points toward a hyperbolic fixed point $z$. Let $\gamma:\R\to X$ be a geodesic flowing at positive time to $z$ with initial condition $\gamma'(0)=v$ and let $\gamma_i$ be the geodesic connecting $z$ with the associated hyperbolic fixed point. By reparametrization we can assume $\gamma_i(+\infty)=z$ and that $\gamma_i(0)$ lies in the same horosphere centered at $z$ than $v$. By estimates in \cite{hi} we have $d(\gamma_i(t),\gamma(t))\to 0$ exponentially fast (here $d$ stands for hyperbolic distance, actually in \cite{hi} the stronger exponential decay in the horospherical distance is obtained). Since the vectors along the geodesics are perpendicular to the horospheres centered at $z$ we have the desired geometric convergence in $TX$. Observe $\gamma_i$ descends to a periodic orbit in $TX/\Gamma$. This gives the convergence of $\gamma$ to the periodic orbit and Birkhoff ergodic theorem gives that the measure generated by such a geodesic is exactly one of $\nu^{h_i}$.

The fact that $h_\nu(g)=0$ for every $\nu\in \M_g\setminus \M_{\Omega_0}$ is a classical result for measures supported on periodic orbits.
\end{proof}

We end up this section with a definition that includes the standing assumptions on the Schottky groups considered in most of our statements. The important point is that we will be in position to use Lemma \ref{lh}, Lemma \ref{lema:mix} and Theorem \ref{s_inftyshottky} below.

\begin{definition} We say a extended Schottky group $\Gamma$ satisfies property $(\star)$ if conditions (C5) and $N_1+N_2\geq 3$ hold.
\end{definition}

\subsection{Geometric meaning of $s_\infty$}

One of our main technical results is the following. Let $\tau$ be the roof function constructed in sub-section \ref{ss:sr}. In the next Theorem we give a geometrical characterisation of the value $s_{\infty}$ defined as the unique real number satisfying
\begin{equation*}
P_\sigma(-t\tau)=
\begin{cases}
\textrm{infinite} & \text{ if } t < s_{\infty}; \\
\textrm{finite} & \text{ if } t >s_{\infty}.\end{cases}
\end{equation*}
One of the key ingredients in this paper, and the important result of this section, is the relation between $s_\infty$ and the largest parabolic critical exponent. This relation will allow us to translate several results at the symbolic level into the geometrical one.

\begin{theorem}\label{s_inftyshottky}
Let $\Gamma$ be an extended Schottky group with property $(\star)$. Let $(\Sigma,\sigma)$ and $\tau:\Sigma \to \R$ be the base space and the roof function of the symbolic representation of the geodesic flow $(g_t)$ on $\Omega_0$. Then $s_{\infty}=\max\{\delta_{p_{i}}, 1\leq i \leq N_{2}\}$.
\end{theorem}
\begin{proof}
We first show that $s_{\infty}\leq \max\{\delta_{p_{i}}, 1\leq i \leq N_{2}\}$.
\begin{eqnarray*}
P_\sigma(-t\tau)&=& \lim_{n \to \infty} \frac{1}{n+1} \log \sum_{x:\sigma^{n+1}x=x} \exp \left(\sum_{i=0}^{n} -t\tau(\sigma^{i}x)\right)  \chi_{C_{h_{1}}}(x)\\
&=& \lim_{n \to \infty} \frac{1}{n+1} \log  \sum_{\xi=\overline{h_{1}x_{2}...x_{n}x_{n+1}}\xi_{0}} \exp \left(\sum_{i=0}^{n} -t B_{\omega^{-1}(\sigma^{i}x)}(o,x_{i+1}o)\right)\\
&\geq& \lim_{n \to \infty} \frac{1}{n+1} \log  \sum_{\xi=\overline{h_{1}x_{2}...x_{n+1}}\xi_{0}} \exp \left(\sum_{i=0}^{n} -t d(o,x_{i+1}o)\right)
\end{eqnarray*}
The last inequality follows from $d(x,y)\ge B_\xi(x,y)$. By removing words having $h_1^m$ (some $m$) in more places than just the first coordinate, we conclude that the argument of the function $\log$ in the limit above is greater than
$$e^{-td(o,h_1o)}\sum_{(c_{1},...,c_{n})\in (\mathcal{A}\setminus h_1)_*^n}\sum_{(m_{1},...,m_{n})\in\Z^n} \exp \left(\sum_{i=1}^{n} -t d(o,c_{i}^{m_{i}}o)\right),$$
where $(\mathcal{A}\setminus h_1)_*^n$ represent the set of admissible words of length $n$ for the code, i.e. $c_i\ne c_{i+1}^{\pm 1}$. Let $k\geq 1$. For all $0\leq j\leq k-1$ and $1\leq i\leq N_{1}+N_{2}-1$, define
$$
b_{i+j(N_{1}+N_{2}-1)} =
\begin{cases}
h_{i+1}, & \text{if } 1\leq i\leq N_{1}-1 \\
p_{i+1-N_{1}}, & \text{if } N_{1} \leq i \leq N_{1}+N_{2}-1.
\end{cases}
$$
Consider $n+1=k(N_{1}+N_{2}-1)$. By restricting the above sum to words with $c_i=b_i$ for every $i=1,...,n$, we can continue the sequence of inequalities above to get
$$P_\sigma(-t\tau) \ge \sum_{m_{1},...,m_{n}\in\Z} \exp \left(\sum_{i=1}^{n} -t d(o,b_{i}^{m_{i}}o)\right),$$
where the right-hand side is equal to
$$\prod_{i=1}^{n}\sum_{m\in\Z}\exp(-td(o,b_{i}^{m}o)).$$
By definition of the $b_{i}$'s, the last term is equal to
$$\Bigg( \prod_{i=2}^{N_{1}}\sum_{m\in\Z}\exp(-td(o,h_{i}^{m}o)) \Bigg)^{k}\Bigg( \prod_{i=1}^{N_{2}}\sum_{m\in\Z}\exp(-td(o,p_{i}^{m}o)) \Bigg)^{k}.$$
Hence, it follows that

\begin{eqnarray*}
P_\sigma(-t\tau)&\geq& \frac{1}{N_{1}+N_{2}} \log \Bigg( \prod_{i=2}^{N_{1}}\sum_{m\in\Z}\exp(-td(o,h_{i}^{m}o)) \Bigg)\Bigg( \prod_{i=1}^{N_{2}}\sum_{m\in\Z}\exp(-td(o,p_{i}^{m}o)) \Bigg)\\
&=& \frac{1}{N_{1}+N_{2}} \log \prod_{a\in\mathcal{A}\setminus h_{1}} P_{<a>}(t,o).
\end{eqnarray*}
In particular, if $t<\max\{\delta_{p_{i}}, 1\leq i \leq N_{2}\}$ then $P_\sigma(-t\tau)=+\infty$. This shows that $s_{\infty}\geq \max\{\delta_{p_{i}}, 1\leq i \leq N_{2}\}$.\\

We prove now the other inequality. Let $(\xi^i_t)$ be the geodesic ray $[o,\omega^{-1}(\sigma^{i+1}x))$. Using (\ref{eq:lem:bus_dist}), we have
\begin{eqnarray*}
\tau(\sigma^{i}x)&=&B_{\omega^{-1}(\sigma^{i}x)}(o,x_i o)\\
&=&B_{\omega^{-1}(\sigma^{i+1}x)}(x^{-1}_i o,o)\\
&=&\lim_{t\to\infty} d(\xi^i_t,x_i o)-d(\xi^i_t,o)\\
&\geq& [d(\xi^i_t,o)+d(o,x_i o)-C]-d(\xi^i_t,o)\\
&=& d(o,x_i o)-C.
\end{eqnarray*}
Thus,
$$\exp(-t\tau(\sigma^{i}x))\leq \exp(tC)\exp(-td(o,x_i o)).$$
Therefore,
\begin{eqnarray*}
P_\sigma(-t\tau)&\leq& \lim_{n \to \infty} \frac{1}{n} \log \sum_{a_{1},...,a_{n}}\sum_{m_{1},...,m_{n}} \prod_{i=1}^{n}\exp(tC)\exp(-td(o,a_i^{m_i} o))\\
&=& \log \left(C^{t}\prod_{a\in\mathcal{A}}P_{<a>}(t,o)\right).
\end{eqnarray*}
In particular, the pressure $P_\sigma(-t\tau)$ is finite for every $t>\max\{\delta_{p_{i}}, 1\leq i \leq N_{2}\}$.
\end{proof}

Denote by $\delta_{p,max}:=\max\{\delta_{p_{i}}, 1\leq i \leq N_{2}\}$. The simplest example to consider is a real hyperbolic space $X$. In this case $\delta_{<p_i>}=1/2$ for any $i\in\{1,...,N_2\}$. In particular $\delta_{p,max}=1/2$. More generally, if we replace hyperbolic space by a manifold of constant negative curvature equal to $-b^2$ then $\delta_{<p>}=b/2$. In the case of non-constant curvature some bounds are known, indeed if the curvature is bounded above by $-a^2$ then $\delta_{<p>} \geq a/2$ (see \cite{dalotpei}).





Recall that at a symbolic level we have $h_{top}(\Phi)=\inf \left\{ t : P_{\sigma}(-t \tau ) \leq 0\right\}$. In particular, if the derivatives of the sectional curvature are uniformly bounded, then Theorem \ref{thm:geod_flow_susp_flow}, Proposition \ref{prop:me} and equality \eqref{eq:OP} imply
\begin{equation}\label{eq:OP_symbolic}
h_{top}(g)= \delta_\Gamma=h_{top}(\Phi).
\end{equation}
The existence of a measure of maximal entropy for the flow $(g_t)$ is related to convergence properties of the Poincar\'e series at the critical exponent. Indeed, using the construction of Patterson and Sullivan (\cite{pa},\cite{su}) of a $\Gamma$-invariant measure on $\partial^2 X$, it is possible to construct a measure on $T^1 X$ which is invariant under the action of $\Gamma$ and the geodesic flow. This measure induces a $(g_t)$-invariant measure on $T^1X/\Gamma$ called the Bowen-Margulis measure. It turns out that, if the group $\Gamma$ is of convergence type then the Bowen-Margulis measure is infinite and dissipative. Hence the geodesic flow does not have a measure of maximal entropy. On the other hand, if the group $\Gamma$ is of divergence type then the Bowen-Margulis measure is ergodic and conservative. If finite, it is the measure of maximal entropy.

It is, therefore, of interest to determine conditions that will ensure that the group is of divergence type and that the Bowen-Margulis measure is finite. It is along these lines that Dal'bo, Otal and Peign\'e \cite{dalotpei} introduced the following:

\begin{definition} \label{def:pgc}
A geometrically finite  group $\Gamma$ satisfies the \emph{parabolic gap condition (PGC)} if its critical exponent $\delta_{\Gamma}$ is strictly greater than the one of each of its parabolic subgroups.
\end{definition}

It was shown in \cite[Theorem A]{dalotpei} that if a group satisfies the PGC-condition then the group is divergent and the measure of Bowen-Margulis is finite  \cite[Theorem B]{dalotpei}. In particular it has a  measure of maximal entropy.
Note that a divergent group in the case of constant negative curvature satisfies the PGC-property.

In our context, an extended Schottky group is a geometrically finite group such that all the parabolic subgroups have rank 1. Moreover, by Condition (C4) and Theorem \ref{thm.DP.1}, it satisfies the PGC-condition. Thus, the following property is a direct consequence of Theorem \ref{s_inftyshottky}, Theorem \ref{thm.DP.1} and \eqref{eq:OP_symbolic}.
\begin{proposition}\label{s_inf_leq_entropy}
Let $X$ be a Hadamard manifold with pinched negative sectional curvature and let $\Gamma$ be an extended Schottky group. Assume that the derivatives of the sectional curvature are uniformly bounded. If $(Y, \Phi)$ is the symbolic representation of the geodesic flow on $T^1X/\Gamma$, then $s_{\infty} < h_{top}(\Phi)$.
\end{proposition}

\section{Escape of Mass for geodesic flows}\label{e:m}

This section contains our main results relating the escape of mass of a sequence of invariant probability measures for a class of geodesic flows defined over non-compact manifolds. We prove that there is a uniform bound, depending only in the entropy of a measure, for the amount of mass a measure can give to the cusps. We also characterise the amount of entropy that the cusp can have.  These results are similar in spirit to those obtained in \cite{ekp, elmv, klkm} for other types of flows.

\begin{theorem} \label{thm:ent_cusp_flow_symbolic_model}
Let $X$ be a Hadamard manifold with pinched negative sectional curvature and let $\Gamma$ be an extended Schottky group of isometries of $X$ with property $(\star)$. Assume that the derivatives of the sectional curvature are uniformly bounded. Then, for every $c>\delta_{p,max}$ there exists a constant $M=M(c)>0$ such that for every $\nu \in \M_{\Omega_0}$ with $h_\nu(g)\geq c$, we have
\begin{equation*}
\int \tau d \mu \leq M,
\end{equation*}
where $\nu$ has a symbolic representation as  $(\mu \times m)/((\mu \times m)(Y))$. Moreover, the value $\delta_{p,max}$ is optimal in the following sense: there exists a sequence $\nu_n \in \M_{\Omega_0}$ of $g$-invariant probability measures such that $\lim_{n \to \infty} \int \tau d \mu_n= \infty$ and
\[\lim_{n \to \infty} h_{\nu_n}(g) =  \delta_{p, max}.\]
\end{theorem}
\begin{proof}
This is a direct consequence of Theorems \ref{thm:ent_cusp} and \ref{thm:ec} using the symbolic model for the geodesic flow on $T^1X/\Gamma$.
\end{proof}

The following corollary is just an equivalence of the first conclusion in Theorem \ref{thm:ent_cusp_flow_symbolic_model} (see also Theorem \ref{thm:ec}).

\begin{corollary}\label{cor:ent_cusp_flow}
Assume $X$ and $\Gamma$ as in Theorem \ref{thm:ent_cusp_flow_symbolic_model}. If $\nu_n \in \M_{\Omega_0}$ is a sequence of $(g_t)$-invariant probability measures such that $\lim_{n \to \infty} \int \tau d \mu_n= \infty$, then
\[\limsup_{n \to \infty} h_{\nu_n}(g) \leq  \delta_{p, max}.\]
\end{corollary}

We are now in position to prove the main result about escape of mass.

\begin{theorem-non}[\ref{thm:escape_mass}] Let $X$ be a Hadamard manifold with pinched negative sectional curvature and let $\Gamma$ be an extended Schottky group of isometries of $X$ with property $(\star)$. Assume that the derivatives of the sectional curvature are uniformly bounded. Then, for every $c>\delta_{p,max}$ there exists a constant $m=m(c)>0$, with the following property: If $(\nu_n)$ is a sequence of ergodic $(g_t)$-invariant probability measures on $T^1X/\Gamma$ satisfying $h_{\nu_n}(g)\geq c$, then for every vague limit $\nu_n \rightharpoonup \nu$, we have
$$\nu(T^1X/\Gamma)\geq m.$$
In particular if $\nu_n \rightharpoonup 0$ then $\limsup h_{\nu_n}(g)\le \delta_{p,max}$. Moreover, the value $\delta_{p,max}$ is optimal in the following sense: there exists a sequence $(\nu_n)$ of $(g_t)$-invariant probability measures on $T^1X/\Gamma$ such that $h_{\nu_n}(g)\to \delta_{p, max}$ and $\nu_n \rightharpoonup 0$.
\end{theorem-non}

\begin{proof} Since every ergodic measure in $\M_g\setminus \M_{\Omega_0}$ has zero entropy, we can suppose that $\nu_n$ belongs to $\M_{\Omega_0}$ for every $n\in\N$. Observe now that the cross-section $S\subset T^1X/\Gamma$ defined in the proof of Theorem \ref{thm:geod_flow_susp_flow} is bounded. Hence, using the identification $\Psi:\Omega_0\to Y$ given by Theorem \ref{thm:geod_flow_susp_flow} and fixing $0<r\leq \inf_{x\in\Sigma} \tau(x)$, there exists a compact set $K_r\subset T^1X/\Gamma$ such that
$$\Sigma\times [0,r]/\sim  \ \subset \Psi(K_r).$$
Let $\mu_n$ be the probability measure on $\Sigma$ associated to the symbolic representation of $\nu_n$. By Theorem \ref{thm:ent_cusp_flow_symbolic_model}, we have
$$\int \tau d\mu_n \leq M.$$
Hence,
\begin{eqnarray*}
\nu_n(K_r) &=& \Psi_{*}\nu_n (\Psi(K_r)) \geq \Psi_{*}\nu_n (\Sigma\times [0,r]/\sim)\\
&=& \frac{\int_\Sigma \int_0^r dtd\mu_n}{\int \tau d\mu_n}\\
&\geq& \frac{r}{M}.
\end{eqnarray*}
In other words, every vague limit $\nu$ of the sequence of ergodic probability measures $(\nu_n)_n$ satisfies $\nu(K_r)\geq r/M$. In particular, we obtain $\nu(T^1X/\Gamma)\geq r/M$. By setting $m=\inf_{x\in\Sigma} \tau(x)/M$, the conclusion follows.\\

Before giving the proof of the optimality of $\delta_{p,max}$, we need the following result.

\begin{proposition}\label{approximation} Let $\Gamma$ be an extended Schottky group of isometries of $X$. Let $p\in\mathcal{A}$ be a parabolic isometry. We can choose a hyperbolic isometry $h\in\Gamma$ for which the groups $\Gamma_n=<p,h^{n}>$ satisfy the following conditions:
\begin{enumerate}
\item The group $\Gamma_{n}$ is of divergence type for every $n\geq 1$,
\item The sequence $(\delta_{\Gamma_{n}})_n$ of critical exponents satisfy $\delta_{\Gamma_{n}}\rightarrow \delta_{\P}$ as $n$ goes to $\infty$,
\item The following limit holds
\begin{equation*}
\lim_{n\to\infty}\frac{n}{\sum_{\gamma\in\P} e^{-\delta_{\Gamma_n} d(x,\gamma x)}}=0.
\end{equation*}
\end{enumerate}
\end{proposition}

\begin{proof}
The proof is based on that of \cite[Theorem C]{dalotpei}. Let $G$ be a group, we will use the notation $G^*$ for $G\setminus\{id\}$. Define $\P=<p>$ and take $U_\P\subset X\cup \partial X$ a connected compact neighbourhood of the fixed point $\xi_p$ of $p$ such that for every $m\in\Z^{\ast}$ we have $p^{m}(\partial X\setminus U_{\P})\subset U_{\P}$. We could  take $U_\P$ so that $U_\P\cap \partial X$ is a fundamental domain for the action of $\P$ in $\partial X\setminus \{ \xi_p \}$. Because $\Gamma$ is non-elementary and $\Lambda_\Gamma$ is not contained in $U_\P$, it is possible to choose $h\in\Gamma$ a hyperbolic isometry of $X$ such that its two fixed points $\xi_{h^-},\xi_{h}$ do not lie in $U_{\P}$. We have used the fact that pair of points fixed by a hyperbolic isometry is dense in $\Lambda\times \Lambda$. Fix $x\in X$ over the axis of $h$. Since $\Gamma$ is an extended Schottky group, for every $k \in \N$ the elements $p$ and $h^k$ are in Schottky position. In particular, for $H_k=<h^k>$ we can find a compact subset $U_{H_k} \subset X\cup \partial X$  satisfying the following three conditions
\begin{enumerate}
\item$H^*(\partial X\setminus U_{H_k})\subset U_{H_k}$.
\item $U_{H_k}\cap U_{\P}=\emptyset$.
\item  $x\not\in U_{H_k}\cup U_{\P}$.
\end{enumerate}

Since $\P$ and $U_{H_k}$ are in Schottky position it is a consequence of the Ping Pong Lemma that the group generated by them is a free product.  By the same argument to that one of Lemma \ref{lem:bus_dist_comparison}, there is a positive constant $C \in \R$ such that for every $y\in U_{H_k}$ and $z\in U_{\P}$, we have
\begin{equation} \label{eq:tri}
d(y,z)\ge d(x,y)+d(x,z)-C.
\end{equation}
 Applying inequality  \eqref{eq:tri} and the inclusion properties described above we obtain
\begin{equation} \label{eq:tria}
d(x,p^{m_1}h^{kn_1}..p^{m_j}h^{kn_j}x)\ge \sum_i d(x,p^{m_i}x)+ \sum_i d(x,h^{kn_i}x)-2kC,
\end{equation}
where $m_i\in \Z^*$. As remarked in \cite{dalotpei} the sum $$\tilde{P}(s)=\sum_{j\ge 1}\sum_{n_i,m_i\in\Z^*}\exp(-sd(x,p^{m_1}h^{kn_1}...p^{m_j}h^{kn_j}x)),$$
is comparable with the Poincar\'e series of $\Gamma_k$. Indeed, since $h$ is hyperbolic both have the same critical exponent. Using the inequality \eqref{eq:tria} we obtain
 \begin{equation*}
 \tilde{P}(s)\le \sum_{j\ge 1} \left(e^{2sC} \sum_{n\in\Z^*}e^{-sd(x,h^{kn}x)}   \sum_{m\in\Z^*}e^{-sd(x,p^m x)}\right)^j .
 \end{equation*}
Because of our  choice of $x$, if $l:=d(x,hx)$ then $d(x,h^{N}x)=|N| l$ for all $N\in\Z$. Thus,
$$\sum_{n\in\Z^*}e^{-sd(x,h^{kn}x)}\le 2\dfrac{e^{-slk}}{1-e^{-slk}}. $$
Let $s_\epsilon:=\delta_{\P}+\epsilon>\delta_{\P}$ and denoting $P_s=\sum_{m\in\Z}e^{-(\delta_\P+s) d(x,p^m x)}$, then the sum $P_\epsilon$ is finite. Assuming $\epsilon$ small, we get a constant $D$ such that
\begin{equation*}
e^{2s_{\epsilon}C} 2\dfrac{e^{-s_{\epsilon}kl}}{1-e^{-s_{\epsilon}kl}}\sum_{m\in\Z^*}e^{-s_{\epsilon}d(x,p^m x)}<De^{-s_{\epsilon}kl}P_{\epsilon}.
\end{equation*}
Hence, if $\log(DP_{\epsilon})/s_\epsilon l<k$, then $De^{-s_{\epsilon}kl}P_{\epsilon}<1$ and therefore $\delta_{\Gamma_k}\le s_\epsilon$. Observe that the function $t\mapsto \log(DP_t)/s_t $ is continuous, decreasing and unbounded in the interval $(0,\eta)$, for any $0<\eta\le1$. We can then solve the equation $\log(DP_t)/s_t l=k-1$, where $t\in (\delta_p,\delta_p+\eta)$ and $k$ is large enough. We call this solution $\epsilon_k$. By construction $\delta_{\Gamma_k}\le s_{\epsilon_k}$. It follows from the definition of $\epsilon_k$ that $\lim_{n\to \infty} P_{\epsilon_k}=\infty$. Observe that
\begin{equation*}
\dfrac{k}{\sum_{\gamma\in\P} e^{-\delta_{\Gamma_k} d(x,\gamma x)}}\le
\dfrac{k}{P_{\epsilon_k}}=\dfrac{\log (DP_{\epsilon_k})/(s_{\epsilon_k}l)+1}{P_{\epsilon_k}},
\end{equation*}
but the RHS goes to 0 as $k\to \infty$. Since $p$ is of divergence type, it follows from \cite[Theorem A]{dalotpei} that $\Gamma_n$ is of divergence type.
\end{proof}

We proceed to show an explicit family of measures satisfying the property claimed in the second part of Theorem \ref{thm:escape_mass}. We remark that the  measures constructed in Theorem \ref{thm:ent_cusp_flow_symbolic_model} can not be used at this point, since a compact set in $T^1X/\Gamma$ is not necessarily a compact set in the topology of $Y$. Hence, the fact that $\int \tau d\mu_n \rightarrow \infty$ does not imply that $\nu_n \rightharpoonup 0$. Despite this difficulty, we can use the geometry to construct the desired family.

Denote by $p$ a parabolic isometry  in the generator set $\mathcal{A}$ with maximal critical exponent, that is $\delta_{p, max}=\delta_{p}$. Take $\Gamma_n=<p,h^{n}>$ as in Proposition \ref{approximation}. Let $m^{BM}_{n}$ be the Bowen-Margulis measure on $T^{1}X/\Gamma_n$. Since an extended Schottky group is a geometrically finite group, the measure $m^{BM}_{n}$ is finite \cite[Theorem B]{dalotpei}. Moreover, it maximises the entropy of the geodesic flow on $T^{1}X/H_{n}$ \cite[Theorem 2]{op}. In other words $h_{m^{BM}_{n}}(g)$ equals $\delta_{\Gamma_n}$. Recall that the critical exponent $\delta_{\Gamma_{n}}$ converges to $\delta_{p, max}$ as $n$ goes to infinity, therefore
\begin{equation}\label{eq:dissipation_1}
h_{m^{BM}_{n}}((g_{t})) \rightarrow \delta_{p, max}.
\end{equation}

Using the coding property, we know that $T^{1}X/\Gamma_n$ (except vectors defining geo\-desics pointing to the $\Gamma_n$-orbit of the fixed points of $h$ and $p$) is identified with $Y_{n}=\{(x,t)\in \Sigma_n \times \R : 0\leq t\leq \tau(x)\}/\sim$, where
$$\Sigma_n= \{(a_{i}^{m_{i}})_{i\in
\Z}:a_{i}\in\{p,h^{n}\}, m_{i}\in\Z \},$$
and the geodesic flow is conjugated to the suspension flow on $(Y_{n},\tau)$ (same $\tau$ as before, but for this coding). It is convenient to think $(\Sigma_n,\sigma)$ as a sub-shift of $(\Sigma,\sigma)$. Since the Bowen-Margulis measure $m_n^{BM}$ is ergodic and has positive entropy it need to be supported in $Y_n$ under the corresponding identification, i.e. in the space of geodesics modeled by the suspension flow. In particular we can consider $m_n^{BM}$ as supported in some invariant subset of $Y$. Let us call $\nu^{BM}_n$ the image measure of $m_n^{BM}$ induced by the inclusion $Y_n \hookrightarrow Y$ and normalized so that $\nu^{BM}_n$ is a probability measure. Observe that \eqref{eq:dissipation_1} implies that
$$\lim_{n\to\infty} h_{\nu_n}(g)=\delta_{p, max}.$$
We just need to prove that $\nu_n^{BM}\rightharpoonup 0$ to end the proof of Theorem \ref{thm:escape_mass}. This sequence actually dissipates through the cusp associated to the parabolic element $p$. Recall that $\xi_p$ denotes the fixed point of $p$ at infinity. Define $N_{\xi_p}(s):=\{x\in X: B_{\xi_p}(o,x)>s\}$, where $o\in X$ is a reference point. Since $\Gamma$ is geometrically finite, for $s$ large enough $N_{\xi_p}(s)/<p>$ embeds isometrically into $T^1X/\Gamma$, i.e. it is a standard model for the cusp at $\xi_p$. By definition, the group $<p>$ acts co-compactly on $\Lambda_\Gamma \setminus\{\xi_p\}$. In other words, if we consider a fundamental domain for the action of $\P$ on $\Lambda_\Gamma \setminus\{\xi_p\}$, say $D$, then $\Lambda_\Gamma \bigcap D$ is relatively compact in $D$. Clearly the other fundamental domains are given by $\gamma D$ where $\gamma\in \P$.

 In \cite{dalotpei} it is proven that for any geometrically finite group $\Gamma$ the Bowen-Margulis measure in the cusp $C$ satisfies a bound of the type
\begin{equation}\label{eq:16}\dfrac{1}{A_{\Gamma,C}}\sum_{p\in\P} d(x,px)e^{-\delta_\Gamma d(x,px)}\le m^\Gamma_{BM}(T^1C)\le A_{\Gamma,C}\sum_{p\in\P} d(x,px)e^{-\delta_\Gamma d(x,px)}.\end{equation}
Here the point $x$ is chosen inside $C$ and the constant $A_{\Gamma,C}$ basically depend on the size of $C$ and the minimal distance between $\Lambda_\Gamma \bigcap D$ and $\partial D$.

Define $\Q_i= N_{\xi_p}(s_i)/<p>$, where the sequence $\{s_i\}_{i\ge 1}$ is  increasing with $\lim_i s_i=\infty$. We assume $Q_1$ provides a standard cusp neighborhood. Denote by $p_n$ the projection
$$p_n:T^1 X/\Gamma_n\to T^1X/\Gamma,$$
induced by the inclusion at the level of groups. By definition
$$\nu_n^{BM}=\dfrac{1}{m_n^{BM}(T^1X/\Gamma_n)}(p_n)_* m_n^{BM}.$$
We will prove that $\lim_{n\to\infty} \nu_n^{BM}(T^1((X/\Gamma )\setminus \Q_i))=0$ for any $i$. For this it is enough to prove the limit
$$\lim_{n\to\infty} \dfrac{m^{BM}_n(p_n^{-1}T^1((X/\Gamma)\setminus \Q_i))}{m^{BM}_n(p_n^{-1}T^1 \Q_i)}=0.$$
Observe that, if $\pi_n:X/\Gamma_n \to X/\Gamma$ is the natural projection, then the sets $\pi_n^{-1}\Q_i$ are represented by the same one in the universal covering. We denote $S_i$ this cusp neighborhood.

\begin{lemma}\label{dissipation1} The measure $m^{BM}_n(p_n^{-1}T^1((X/\Gamma)\setminus \Q_i)))$ growth at most linearly in $k_n$, that is for certain positive constant $C_i$, we have
\begin{equation}\label{eq:dissipation1}
m^{BM}_n(p_n^{-1}T^1((X/\Gamma) \setminus \Q_i))\le C_i n.
\end{equation}
\end{lemma}
\begin{proof} Let $D_0$ (resp. $D_n$) be the fundamental domain of $\Gamma$ (resp. $\Gamma_n$) containing $o\in X$. By definition of fundamental domain, there exists a set $T_n \subset \Gamma$ such that
\begin{enumerate}
\item for any $\gamma_1,\gamma_2 \in T_n$ and $\gamma_1\ne\gamma_2$, we have $\gamma_1 \text{int}(D_0) \cap \gamma_2 \text{int}(D_0) = \emptyset$, and
\item $\bigcup_{\gamma\in T_n} \gamma D_{0} = D_n$.
\end{enumerate}
Denote by $K_i$ the compact $K_i=(X/\Gamma) \setminus \Q_i$ and let $\widetilde{K}_i$ the lift of $K_i$ into $X$ intersecting $D_0$. By definition, any lift of $K_i$ into $X$ intersecting $D_n$ is a translation of $\widetilde{K}_i$ by an element in $T_n$. Since $m^{BM}_n$ is supported in $\Lambda_{\Gamma_n}$, the Bowen-Margulis measure $\tilde{m}^{BM}_n$ on $X$ satisfies
\begin{equation*}
m^{BM}_n(p_n^{-1}T^1(K_i)) \leq \sum_{\stackrel{\gamma\in T_n}{\gamma \widetilde{K}_i \cap C(\Gamma_n)\neq \emptyset}} \tilde{m}^{BM}_n(T^1(\gamma \widetilde{K}_i)),
\end{equation*}
where $C(\Gamma_n)$ is the convex hull of $L(\Gamma_n)\times L(\Gamma_n)$ in $X\cup \partial X$. By construction and convexity of the domains $C_\gamma$, there exists a positive integer $N\geq 1$ such that
$$\#\{\gamma\in T_n : \gamma \widetilde{K}_i \cap C(\Gamma_n) \neq \emptyset \} \leq N n.$$
In particular, we have
\begin{equation*}
m^{BM}_n(p_n^{-1}T^1(K_i)) \leq N n \tilde{m}^{BM}_n(T^1(\widetilde{K}_i)).
\end{equation*}
But again, by estimates given in \cite{dalotpei}, the measure $\tilde{m}^{BM}_n(T^1(\widetilde{K}_i))$ satisfies
$${m}^{BM}_n(T^1(\widetilde{K}_i))\leq L_i,$$
where $L_i$ is a constant depending on the diameter of $\widetilde{K}_i$. By setting $C_i= NL_i$, the conclusion follows.
\end{proof}

Using the comments just below equation \eqref{eq:16} we know that the constants $A_{H_n,\Q_i}$ can be all considered equal to $A_{H_1,\Q_i}$. We have then
\begin{equation}\label{eq:dissipation2}
m^{BM}_n(T^1S_i) \asymp_{A_{H_1,\Q_i}} \sum_{p\in\P} d(x,px)e^{-\delta_{\Gamma_n}d(x,px)}.
\end{equation}
Hence, from \eqref{eq:dissipation1} and \eqref{eq:dissipation2}, we get
\begin{eqnarray*}
\dfrac{m^{BM}_n(p_n^{-1}T^1(X\setminus \Q_i))}{m^{BM}_n(p_n^{-1}T^1 \Q_i)} &\leq& \frac{A_{H_1,\Q_i} C_i n}{\sum_{p\in\P} d(x,px)e^{-\delta_{\Gamma_n} d(x,px)}}\\
&\leq& \frac{C_i' n}{\sum_{p\in\P} e^{-\delta_{\Gamma_n} d(x,px)}}.
\end{eqnarray*}

Finally, property (3) in Proposition \ref{approximation} implies that the last term above converges to 0. Therefore
$$\lim_{n\to\infty} \dfrac{m^{BM}_n(p_n^{-1}T^1(X\setminus \Q_i))}{m^{BM}_n(p_n^{-1}T^1 \Q_i)}=0,$$
which concludes the proof of Theorem \ref{thm:escape_mass}.
\end{proof}

\begin{corollary} Let $X$ and $\Gamma$ as in Theorem \ref{thm:escape_mass}. Then, the entropy at infinity of the geodesic flow satisfies $$h_\infty(T^1X/\Gamma,(g_t))=\delta_{p,max}.$$
\end{corollary}

\section{Thermodynamic Formalism} \label{s:tf}

In this section we always consider $X$ a Hadamard manifold with pinched negative sectional curvature and $\Gamma$ an extended Schottky group of isometries of $X$ with property $(\star)$. We also assume that the derivatives of the sectional curvature are uniformly bounded. Our goal is to obtain several results on thermodynamic formalism for the geodesic flow over $X/\Gamma$. Some of these results were already obtained, without symbolic methods, by Coud\`ene (see \cite{cou}) and Paulin, Pollicott and Schapira (see \cite{pps}). However, the strength of our symbolic approach will be clear in the study of regularity properties of the pressure (sub-section \ref{s_phase}).

Here we keep the notation of sub-section \ref{ss:sr}. Thus, the geodesic flow $(g_t)$ in the set $\Omega_0$ can be coded by a suspension semi-flow $(Y, \Phi)$ with base $(\Sigma, \sigma)$ and roof function $\tau:\Sigma \to \R$.

\subsection{Equilibrium measures}

We will consider the following class of  potentials.
\begin{definition}
A continuous function $f:T^1X/ \Gamma \to \R$ belongs to the class of \emph{regular} functions, that we denote by $\mathcal{R}$, if the symbolic representation $\Delta_f : \Sigma \to \R$ of $f|_{\Omega_0}$ has summable variations.
\end{definition}

We begin studying thermodynamic formalism for the geodesic flow restricted to the set $\Omega_0$. The following results can be deduced from the general theory of suspension flows over countable Markov shifts and from the symbolic model for the geodesic flow. With a slight abuse of notation, using the indentification explicited before, we still denote by $f:Y \to \R$ the given map $f:\Omega_0 \to \R$.

\begin{definition}
Let $f  \in \mathcal{R}$, then the pressure of $f$ with respect to the geodesic flow $g:=(g_t)$ restricted to the set $\Omega_0$ is defined by
\begin{equation*}
P_{\Omega_0}(f):=\lim_{t \to \infty} \frac{1}{t} \log \left(\sum_{\phi_s(x,0)=(x,0), 0<s \leq t} \exp\left( \int_0^s f(\phi_k(x,0)) ~{\rm d}k \right) \chi_{C_{i_0}}(x) \right).
\end{equation*}
\end{definition}

This pressure satisfies the following properties:

\begin{proposition}[Variational Principle] Let $f \in \mathcal{R}$, then
\begin{equation*}
P_{\Omega_0}(f)=\sup \left\{ h_{\nu}(g) +\int_{\Omega_0} f ~{\rm d} \nu : \nu\in
\mathcal{M}_{\Omega_0} \text{ and } -\int_{\Omega_0} f \, ~{\rm d}\nu <\infty \right\},
\end{equation*}
where $\mathcal{M}_{\Omega_0}$ denotes the set of $(g_t)$-invariant probability measures supported in $\Omega_0$.
\end{proposition}

\begin{proposition}
Let $f \in \mathcal{R}$, then
\begin{eqnarray*}
P_{\Omega_0}(f) = \sup \{ P_{g|K}(f) : K\in \cK_{\Omega_0}(g) \},
\end{eqnarray*}
where $\cK_{\Omega_0}(g)$ denotes the space of compact $g-$invariant sets in $\Omega_0$.
\end{proposition}

\begin{remark}[Convexity]
It is well known that for any $ K\in \cK_{\Omega_0}(g)$ the pressure function $ P_{g|K}(\cdot)$ is convex.  Since the supremum of convex functions is a convex function, it readily follows that $P_{\Omega_0}(\cdot)$ is convex.
\end{remark}

\begin{proposition}\label{prop_es}
Let $f \in \mathcal{R}$. Then there is an equilibrium measure $\nu_f \in \M_{\Omega_0}$, that is,
\begin{equation*}
P_{\Omega_0}(f)= h_{\nu_f}(g) +\int_{\Omega_0} f ~{\rm d} \nu_f,
\end{equation*}
for $f$ if  and only if
we have that $P_\sigma(\Delta_f -P_{\Phi}(f) \tau)=0$ and there exists
 an equilibrium measure $\mu_f\in \M_\sigma$ for $\Delta_f -P_{\Phi}(f) \tau$ such that $\int \tau  d\mu_f < \infty$.
 Moreover, if such an equilibrium measure exists then it is unique.
\end{proposition}

In order to extend these results to the geodesic flow in $T^1X/ \Gamma$ we use the second conclusion of Proposition \ref{prop:me}.

\begin{definition} \label{def:press}
Let $f \in \mathcal{R}$, then the pressure of $f$ with respect to the geodesic flow $g:=(g_t)$ in $T^1X/ \Gamma$ is defined by
\begin{equation*}
P_g(f):=\max \left\{ P_{\Omega_0}(f) , \int f ~{\rm d} \nu^{h_1}, \dots , \int f ~{\rm d} \nu^{h_{N_1}}   \right\}.
\end{equation*}
\end{definition}

\begin{proposition}[Variational Principle] Let $f \in \mathcal{R}$, then
\begin{equation*}
P_g(f)= \sup \left\{ h_{\nu}(g) +\int f ~{\rm d} \nu : \nu\in
\mathcal{M}_{g} \text{ and } -\int f \, ~{\rm d}\nu <\infty \right\},
\end{equation*}
\end{proposition}

\begin{proposition}
Let $f \in \mathcal{R}$, then
\begin{eqnarray*}
P_{g}(f) = \sup \{ P_{g|K}(f) : K\in \cK(g) \},
\end{eqnarray*}
where $\cK(g)$ denotes the space of compact $g-$invariant sets.
\end{proposition}

\begin{proposition}
Let $f \in \mathcal{R}$ be such that $\sup f < P_g(f)$ then  there is an equilibrium measure $\nu_f \in \M_{g}$, for $f$ if  and only if
we have that $P_{\Omega_0}(\Delta_f -P_{\Phi}(f) \tau)=0$ and there exists
 an equilibrium measure $\mu_f\in \M_\sigma$ for $\Delta_f -P_{\Phi}(f) \tau$ such that $\int \tau  d\mu_f < \infty$.
 Moreover, if such an equilibrium measure exists then it is unique.
\end{proposition}

\begin{proof}
Note that if $\sup f < P_g(f)$ then an equilibrium measure for $f$, if it exists, must have positive entropy. Since the measures $\nu^{h_i}$, with $i \in \{1, \dots , N_1\}$ have zero entropy (see Proposition \ref{prop:me}). The result follows from Proposition \ref{prop_es}.
 \end{proof}

The next result shows that potentials with small oscillation do have equilibrium measures, this result can also be deduced from  \cite{cou,pps}. Our proof is short and uses the symbolic structure.

\begin{theorem} \label{t:osc}
Let $f \in \mathcal{R}$. If
\begin{equation*}
\sup f - \inf f < h_{top}(g) - \delta_{p, max}
\end{equation*}
then $f$ has an equilibrium measure.
\end{theorem}

\begin{proof}
Assume that the measures $\nu^{h_i}$ are not equilibrium measures for $f$, otherwise the theorem is proved.  Therefore, we have that
$P_g(f)=P_{\Omega_0}(f)$. We  first show  that $P_\sigma(\Delta_f -P_g(f) \tau)=0$. Note that for every $x \in \Sigma$,
\begin{equation*}
 \tau(x)  \inf f  \leq \Delta_f(x) \leq  \tau(x)  \sup f.
\end{equation*}
By monotonicity of the pressure we obtain
\begin{equation*}
P_\sigma((\inf f -t )\tau) \leq   P_\sigma(\Delta_f -t \tau) \leq P_\sigma((\sup f -t )\tau).
\end{equation*}
Let $t \in (s_{\infty} + \sup f, h_{top}(g) + \inf f)$ and recall that $s_{\infty} = \delta_{p, max}$. Then
\begin{equation*}
0<P_\sigma((\inf f -t)\tau) \leq   P_\sigma(\Delta_f -t \tau) \leq P_\sigma((\sup f -t )\tau)< \infty.
\end{equation*}
Since $P_g(f)<\infty$ and the function $t \to  P_\sigma(\Delta_f -t \tau)$ is continuous with $\lim_{t\to\infty}  P_\sigma(\Delta_f -t \tau)=-\infty$,  we obtain that $P_\sigma(\Delta_f -P_g(f) \tau)=0$.
Since the system $\Sigma$ has the BIP condition and the potential $\Delta_f -P_g(f) \tau$ is of summable variations, it has an equilibrium measure $\mu$.
It remains to prove the integrability condition. Recall that
\[\frac{\partial}{\partial t} P_\sigma(\Delta_f -t \tau) \Big|_{t=P_g(f)} = -\int\tau d \mu.\]
But we have proved that the function $t \to P_\sigma(\Delta_f -t \tau)$ is finite (at least) in an interval of the form $[P_g(f)- \epsilon, P_g(f) +\epsilon]$. The result now follows, because when finite the function is real analytic.
\end{proof}

\subsection{Phase transitions} \label{s_phase}
This sub-section is devoted to study the regularity properties of pressure functions $t \mapsto P_g(tf)$ for a certain class of functions $f$. We say that the pressure function $t \mapsto P_g(tf)$ has a \emph{phase transition} at $t=t_0$ if the pressure function is not real analytic at $t=t_0$. The set of points at which the pressure function exhibits phase transitions might be a very large set. However, since the pressure is a convex function  it can only have a countable set of points where it is not differentiable. Regularity properties of the pressure are related to important dynamical properties, for example exponential decay of correlations of equilibrium measures. In the Axiom A case the pressure is real analytic. Indeed, this can be proved noting that, in that setting,  the function $t \mapsto P_\sigma(\Delta_f -t \tau)$ is real analytic and that
$P_\sigma(\Delta_f -P_{\Phi} \tau)=0$. The result then follows from the implicit function theorem noticing that the non-degeneracy condition is fulfilled:
\begin{equation*}
\frac{\partial}{\partial t}  P_{\sigma}(\Delta_f -t \tau) = -\int \tau d \mu <0,
\end{equation*}
where $\mu$ is the equilibrium measure corresponding to $\Delta_f -t \tau$. The inequality above, together with the coding properties established in \cite{bo1, br, ra}, allow us to establish that the pressure is real analytic for regular potentials in the Axiom A setting. In the non-compact case the situation can be different. However, the only results involving the regularity properties of the pressure function for geodesic flows defined on non-compact  manifolds, that we are aware of, are those concerning the modular surface (see \cite[Section 6]{ij}). In this section we establish regularity results for pressure functions of geodesic flows defined on extended Schottky groups. We begin by defining conditions (F1) and (F2) on the potentials.

\begin{definition}
Consider a non-negative continuous function $f:T^1X/ \Gamma \to \R$. We will say $f$ satisfies Condition (F1) or (F2) if the corresponding property below holds.

\begin{enumerate}
\item[(F1)] The symbolic representation $\Delta_f : \Sigma^+ \to \R$ is locally H\"older and bounded away from zero in every cylinder $C_{a^m}\subset\Sigma^+$, where $a\in\A$, $m\in \Z$.
\item[(F2)] Consider any indexation $(C_n)_{n\in\N}$ of the cylinders of the form $C_{a^m}$. Then,
\begin{equation*}
\lim_{n\to \infty} \dfrac{\sup \{\Delta_f(x):x\in C_n\}}{\inf \{\tau(x):x\in C_n\}}=0.
\end{equation*}
\end{enumerate}
We say $f$ belong to the class $\F$ if it satisfies (F1) and (F2).
\end{definition}

In the following Lemma we establish two properties of potentials in $\F$ that will be used in the sequel.
\begin{lemma} \label{p:tau}
Let $f$ be a potential satisfying (F1) and $(\nu_n)$ a sequence of measures in $\M_{\Omega_0}$ such that $\nu_n= \frac{\mu_n \times m|_{Y}}{(\mu_n \times m)(Y)}$. Then,
\begin{enumerate}
  \item[(1)] if  $\lim_{n \to \infty} \int_{\Omega_0} f d\nu_n =0$, then $\lim_{n \to \infty} \int \tau d \mu_n = \infty$.
  \item[(2)] if $f$ satisfies Condition (F2) and $\lim_{n \to \infty} \int \tau d \mu_n = \infty$, then $\lim_{n \to \infty} \int_{\Omega_0} f d\nu_n =0$.
\end{enumerate}
\end{lemma}
\begin{proof}
To prove (1) we will argue by contradiction. Assume, passing to a sub-sequence if necessary, that
\begin{equation*}
\lim_{n \to \infty} \int \tau d \mu_n = C.
\end{equation*}
Let $\epsilon >0$, there exists $N \in \N$ such that for every $n>N$ we have that
\begin{equation*}
\left| 	 \int \tau d \mu_n - C	\right| < \epsilon.
\end{equation*}

\begin{lemma} \label{l:3}
Let $r >1$ then for every $n >N$ we have that
\begin{equation*}
\mu_n(\{x: \tau(x) \leq r\}) > 1- \frac{C+\epsilon}{r}.
\end{equation*}
\end{lemma}

\begin{proof}[Proof of Lemma \ref{l:3}]
Since the function $\tau$ is positive we have
\begin{equation*}
\int \tau d \mu_n \geq r \mu_n (\{x: \tau(x) \geq r\}) + \int_{\{x: \tau(x) \leq r\}} \tau d \mu_n.
\end{equation*}
Thus,
\begin{equation*}
C+ \epsilon \geq r \mu_n (\{x: \tau(x) \geq r\}).
\end{equation*}
\[ \frac{C+\epsilon}{r} \geq  \mu_n (\{x: \tau(x) \geq r\}).\]
Finally
\[\mu_n(\{x: \tau(x) \leq r\}) > 1- \frac{C+\epsilon}{r}.\]
\end{proof}
Note that the set $\{x: \tau(x) \leq r\}$ is contained in a finite union of cylinders on $\Sigma$. This follows from the inequality $d(o,a^m o)-C \leq \tau(x)$, which is a consequence of Lemma \ref{lem:bus_dist_comparison}, and the fact that $\mathcal{A}$ is finite. Since $\Delta_f$ is bounded away from zero in every one of them, there exist a constant $G(r)>0$ such that
\[\Delta_f(x) > G(r),\]
on $\{x: \tau(x) \leq r\}$. Thus
\begin{eqnarray*}
\int_{\Omega_0} f d \nu_n = \frac{\int_{\Sigma} \Delta_f(x) d\mu_n}{\int_{\Sigma} \tau d \mu_n }  \geq
  \frac{\int_{\{x: \tau(x) \leq r\}} \Delta_f(x) d\mu_n}{\int_{\Sigma} \tau d \mu_n }   \geq
  \frac{G(r)\left( 1- \frac{C+\epsilon}{r}\right)}{C-\epsilon}.
\end{eqnarray*}
If we choose $r$ large enough so that $ 1- \frac{C+\epsilon}{r}>0$ we obtain the desired contradiction.

To prove (2) observe that for every $\varepsilon>0$ there exists $N\geq 1$ such that for every $k\geq N$, we have
$$\dfrac{\sup \{\Delta_f(x):x\in C_k\}}{\inf \{\tau(x):x\in C_k\}}<\varepsilon.$$
Hence,
\begin{eqnarray*}
\lim_{n\to\infty} \int f d\nu_n &=& \lim_{n\to\infty} \frac{1}{\int \tau d\mu_n} \sum_{k\geq 1} \int_{C_k} \Delta_f d\mu_n \\
&=& \lim_{n\to\infty} \frac{1}{\int \tau d\mu_n} \sum_{k\geq N} \int_{C_k} \Delta_f d\mu_n \\
&\leq& \lim_{n\to\infty}\frac{1}{\int \tau d\mu_n} \sum_{k\geq N} \int_{C_k} \frac{\sup\{\Delta_f(x):x\in C_k \}}{\inf \{\tau(x):x\in C_k\}}\inf \{\tau(x):x\in C_k\} d\mu_n \\
&\leq& \lim_{n\to\infty}\frac{1}{\int \tau d\mu_n} \sum_{k\geq N} \int_{C_k} \varepsilon \inf \{\tau(x):x\in C_k\} d\mu_n \\
&\leq& \varepsilon.
\end{eqnarray*}
Since $\varepsilon>0$ is arbitrary, it follows the conclusion of the second claim.
\end{proof}

Combining Theorem \ref{thm:ec} and Lemma \ref{p:tau}, we obtain the following

\begin{lemma} \label{l:m}
Let $\Gamma$ be an extended Schottky group with property $(\star)$ and let $f$ a function satisfying property (F1). If
 $(\nu_n)_n \subset \M_{g}$  is a sequence of invariant probability measure for the geodesic flow such that
 \begin{equation*}
 \lim_{n \to \infty} \int_{T^1X/\Gamma} f d \nu_n =0.
 \end{equation*}
Then
 \[ \limsup _{n \to \infty} h_{\nu_n}(g) \leq \delta_{p,max}.\]
\end{lemma}

\begin{proof} If $\nu_n\in \M_{\Omega_0}$ the Lemma follows directly from (1) in Lemma \ref{p:tau} and Theorem \ref{thm:ec}. If $\nu_n\in \M_g\setminus \M_{\Omega_0}$ then we can consider the measure $\tilde{\nu}_n:=\nu_n-\sum_{i=1}^{N_1}c^n_i\nu^{h_i}$ where the constants $c^n_i\ge 0$ are chosen so that $\tilde{\nu}_n$(periodic orbit associated to hyperbolic generator $h_i)=0$. Let $C_n=\tilde{\nu}_n(X/\Gamma)$. If $C_n=0$ then $h(\nu_n)=0$ and there is not contribution to the desired $\limsup h(\nu_n)$. Otherwise define $u_n=C_n^{-1}\tilde{\nu}_n$. By definition $u_n$ is a probability measure in $\M_{\Omega_0}$, we claim $\lim_{n\to \infty} \int fdu_n=0$.
Observe $\int f d\nu_n= \int f d\tilde{\nu}_n+\sum_{i=1}^{N_1}c^n_i\nu^{h_i}(f)$ has non-negative summands and it is converging to zero, it follows that $ \int f d\tilde{\nu_n}\to 0$, $c^n_i\to 0$ as $n\to \infty$. By definition $C_n=1-\sum_{i=1}^{N_1} c^n_i$, therefore $C_n\to 1$. Recalling $u_n=C_n^{-1}\tilde{\nu}_n$ we get $\lim \int fdu_n=0$. Because $\nu_n=C_n u_n+(1-C_n)(1-C_n)^{-1}\sum_{i=1}^{N_1}c^n_i\nu^{h_i}$, we have $h_{\nu_n}(g)=C_n h_{u_n}(g)$. Finally since $C_n\to 1$ and $\limsup_{n\to \infty} h_{u_n}(g)\le \delta_{p,max}$ (because $u_n\in\M_{\Omega_0}$), we get $\limsup_{n\to \infty} h_{\nu_n}(g)\le \delta_{p,max}$.
\end{proof}


The next Theorem is the main result of this sub-section and it is an adaptation of results obtained at a symbolic level in \cite{ij}. It is possible to translate those symbolic results into this geometric setting thanks to Theorem  \ref{s_inftyshottky}.

\begin{theorem-non}[\ref{thm:gpt}]
Let $X$ be a Hadamard manifold with pinched negative sectional curvature and let $\Gamma$ be an extended Schottky group of isometries of $X$ with property $(\star)$. Assume that the derivatives of the sectional curvature are uniformly bounded. If $f \in \mathcal{F}$, then
\begin{enumerate}
\item For every $t \in \R$ we have that $P_g(tf) \geq \delta_{p,max}$.
\item We have that $\lim_{t \to -\infty} P_g(tf)= \delta_{p,max}$.
\item Let $t':= \sup \{t \in\R: P_g(tf) = \delta_{p,max}\}$, then
\begin{equation*}
P_g(tf)=
\begin{cases}
\delta_{p,max} & \text{ if } t <t'; \\
\text{real analytic, strictly convex, strictly increasing} & \text{ if } t>t'.
\end{cases}
\end{equation*}
\item If $t>t'$, the potential $tf$ has a unique equilibrium measure. If $t<t'$ it has no equilibrium measure.
\end{enumerate}
\end{theorem-non}

Note that Theorem \ref{thm:gpt} shows that when $t'$ is finite then the pressure function exhibits a phase transition at $t=t'$ whereas when $t'=-\infty$ the pressure function is real analytic where defined (see Figure below). Recall that $\delta_{p,max}=s_{\infty}$.

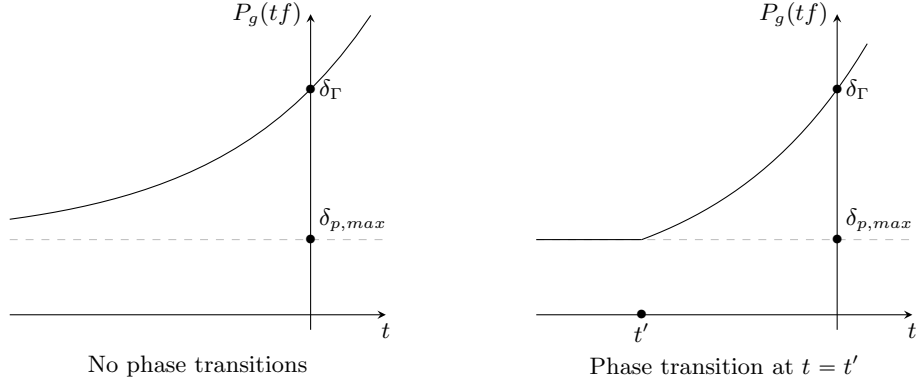
\begin{center}
\begin{tikzpicture}[scale=2]

\small
\draw [>= stealth, ->](0,0)--(2.5,0) node[below=12pt,midway]{$\textmd{No phase transitions}$};
\draw [>= stealth, ->](3.5,0)--(6,0) node[below=12pt,midway]{$\textmd{Phase transition at } t=t'$};

\draw (2.5,0) node[below]{$t$};
\draw (6,0) node[below]{$t$};

\draw [>= stealth, ->](2,-0.1)--(2,2);
\draw [>= stealth, ->](5.5,-0.1)--(5.5,2);

\draw (2,2) node[left]{$P_{g}(tf)$};
\draw (5.5,2) node[left]{$P_{g}(tf)$};

\small

\draw (2,1.5) node{$\bullet$} node[right]{$\delta_{\Gamma}$};
\draw (5.5,1.5) node{$\bullet$} node[right]{$\delta_{\Gamma}$};

\draw[color=gray!50] [dashed](0,0.5)--(2.5,0.5);
\draw[color=gray!50] [dashed](3.5,0.5)--(6,0.5);

\draw (2,0.5) node{$\bullet$} node[above right]{$\delta_{p,max}$};
\draw (5.5,0.5) node{$\bullet$} node[above right]{$\delta_{p,max}$};

\draw [domain=0:2.4]plot(\x,{exp(\x-2)+0.5});

\draw [domain=4.2:5.7]plot(\x,{1.374*exp(\x-5.5)+0.125});
\draw [domain=3.5:4.2]plot(\x,0.5);

\draw (4.2,0) node{$\bullet$} node[below]{$t'$};

\end{tikzpicture}
\end{center}

\begin{proof}[Proof of (1)]
The first claim follows from the variational principle. By Theorem \ref{thm:ent_cusp_flow_symbolic_model} there exists a sequence $(\nu_n)\subset \M_g$ such that $\lim_{n\to\infty} h_{\nu_n}(g)=s_\infty$ and their corresponding probability $\sigma$-invariant measures $(\mu_n)$ in $\Sigma$ satisfy $\lim_{n\to\infty}\int \tau d\mu_n = \infty$. Therefore, by (2) in Lemma \ref{p:tau}, we also have that $\lim_{n\to\infty} \int f d\nu_n = 0$. Hence, for every $t \in \R$, we have
\begin{eqnarray*}
s_{\infty}= \delta_{p, max}&=& \lim_{n \to \infty} \left(h_{\nu_n}(g) + t \int_{\Omega_0} f d \nu_n  \right)\\
&\leq& \sup \left\{ h_{\nu}(g) + t \int f d \nu : \nu \in \M_g \right\} =P_g(tf).
\end{eqnarray*}
\end{proof}

\begin{proof}[Proof of (2)] Since $t\mapsto P_g(tf)$ is non-decreasing and bounded below, the following limit $\lim_{t \to -\infty} P_g(tf)$ exists. Define $A\in\R$ as the limit $\lim_{t \to -\infty} P_g(tf):=A$. Using the Variational Principle, we can choose a sequence of measures $(\nu_n)_n$ in $\M_g$ for which
$$\lim_{n\to\infty} h_{\nu_n}(g)-n\int f d\nu_n = A.$$
Since $A$ is finite it follows that $\lim_{n \to \infty}  \int f d\nu_n=0$. Hence, from Lemma \ref{l:m}, we obtain $\limsup _{n \to \infty} h_{\nu_n}(g) \leq  s_{\infty}$. In particular,
\begin{eqnarray*}
s_\infty &\leq& \lim_{t \to -\infty} P_g(tf) \\
&=& \lim_{n\to\infty} h_{\nu_n}(g)-n\int f d\nu_n\\
&\leq&  \lim_{n\to\infty} h_{\nu_n}(g) \leq s_{\infty}.
\end{eqnarray*}
Therefore, we have that $A=\delta_{p, max}$.
\end{proof}

\begin{proof}[Proof of (3): Real analyticity]
We first prove $P_g(tf)=P_{\Omega_0}(tf)$. After this is done we can proceed with standard regularity arguments in the symbolic picture. Observe that for $t<0$ the pressure $P_{\Omega_0}(tf)$ is always positive while the contribution of the pressure on $(T^{1}X/\Gamma) \setminus \Omega_0$ is negative, so $P_g(tf)=P_{\Omega_0}(tf)$ for every $t\leq 0$. Consider now $t>0$. Pick $\nu^{h_i}$ as in Proposition \ref{prop:me} (see also Definition \ref{def:press}). Denote by $x_-$ (resp. $x_+$) the repulsor (attractor) of $h_i$ and $\gamma_{h_i}$ the geodesic defined by those points. Consider $p$ a parabolic element in $\A$ and let $\gamma_n$ be the geodesic connecting the points $\xi^-$ and $\xi^+$ where $\omega(\xi^-)=\overline{p^{-1}h^{-n}}$ and $\omega(\xi^+)=\overline{h^np}$. Denote $\gamma_\infty$ the geodesic connecting $p^{-1}x_-$ and $x_+$. Observe $\gamma_n$ descends to a closed geodesic in  $T^1X/\Gamma$. By comparing $\gamma_n$ and $\gamma_\infty$ we see that for any $\epsilon>0$, the amount of time $\gamma_n$ leaves a $\epsilon$-neighborhood of $\gamma_{h_i}$ is uniformly bounded for big enough $n$. Let $\nu_n$ be the invariant probability measure defined by the closed geodesic $\gamma_n$, then we get the weak convergence $\nu_n\to \nu^{h_i}$. Then
$$t\int f d\nu^{h_i}=\lim_{n\to\infty}t\int f d\nu_n\le \lim_{n\to\infty} (h_{\nu_n}(g)+t\int f d\nu_n)\le P_{\Omega_0}(tf).$$
This give us $P_g(tf)=P_{\Omega_0}(tf)$.

The pressure function $t \mapsto P_g(tf)$ is convex, non-decreasing and bounded from below by $s_{\infty}$. We now prove that
for $t>t'$ it is real analytic. Note that since $t'<t$ we have that
\[P_\sigma(t\Delta_f -s_{\infty}\tau)>0,\]
possibly infinity and that there exists $p>s_{\infty}$ such that $0<P_\sigma(t\Delta_f -p\tau)<\infty$ (see \cite[Lemma 4.2]{ij}). Moreover, Condition $(F2)$ implies that $P_g(tf)<\infty$ for every $t>t'$, hence
\[P_\sigma(\Delta_{tf} -P_g(tf) \tau) \leq 0.\]
Since $\tau$ is a positive, the function $s \mapsto P_\sigma(t\Delta_f - s\tau)$ is decreasing and
$$\lim_{s \to +\infty}  P_\sigma(t\Delta_f - s \tau) = -\infty.$$
Moreover, since the base map of the symbolic model satisfies the BIP condition, the function $(s,t)\mapsto P_\sigma(t\Delta_f - s\tau)$ is real analytic in both variables. Hence, there exists a unique real number $s_f>s_{\infty}$  such that $P_\sigma(t\Delta_f -s_f \tau)=0$  and
$$\frac{\partial}{\partial s} P_\sigma(t\Delta_f - s\tau) \Big|_{s=s_f} <0.$$
Therefore, $P_g(tf)=s_f$ and by Implicit Function Theorem, the function $t \mapsto P_g(tf)$ is real analytic in $(t', t^\star)$.
\end{proof}

\begin{proof}[Proof of (4)]
First note that the previous claims imply that no zero entropy measure can be an equilibrium measure. Moreover, in the proof of (3) we obtained that for $t \in (t', \infty)$ we have that $P_\sigma(t\Delta_f -P_g(tf) \tau)=0$. Since the system satisfies the BIP condition there exists an equilibrium measure $\mu_f \in \M_{\sigma}$ for $t\Delta_f -P_g(tf) \tau$ such that $\int \tau d \mu_f <\infty$ (see Theorem \ref{thm_es}). Therefore it follows from Proposition \ref{prop_es} that $tf$ has an equilibrium measure.

In order to prove the last claim, assume by contradiction that for some $t_1 <t'$ the potential $t_1f$ has an equilibrium measure $\nu_{t_1}$. Then $s_{\infty}= P_g(t_{1}f)= h_{\nu_{t_{1}}}(g) + t_{1} \int_{\Omega_0} f d \nu_{t_{1}}$. Since $f>0$ on $\Omega_0$, we have that $\int_{\Omega_0} f d \nu_{t_{1}} :=B >0$.  Thus the straight line $r \to  h_{\nu_{t_{1}}}(g) + r \int_{\Omega_0} f d \nu_{t_{1}}$ is increasing with $r$, therefore for $t \in (t_{1}, t')$ we have that
\begin{equation*}
h_{\nu_{t_{1}}}(g) + t \int_{\Omega_0} f d \nu_{t_{1}} > s_{\infty}=P_g(tf).
\end{equation*}
This contradiction proves the statement.
\end{proof}

\subsection{Examples}

We will use the following criterion, first introduced in \cite{ij}, to construct phase transitions.

\begin{proposition}\label{phtr} Let $f \in \F$. Then
\begin{enumerate}
\item If there exist $t_0\in\R$ such that $P_\sigma(t_0\Delta_f-s_\infty \tau)<\infty$, then  there exists $t'<t_0$ such that for every $t <t'$ we have $P_g(tf)=s_\infty$.
\item  Suppose that there exists an interval $I$ such that $P_\sigma(t\Delta_f-s_\infty \tau)=\infty$ for every $t\in I$. Then $t\mapsto P_g(tf)$ is real analytic on $I$. In particular, if for every $t\in\R$ we have $P_\sigma(t\Delta_f-s_\infty \tau)=\infty$, then $t\mapsto P_g(tf)$ is real analytic in $\R$.
\end{enumerate}\end{proposition}

The proof of this Lemma follows as in \cite[Lemma 4.5, Theorem 4,1]{ij}. We now present an example of a phase transition (Example \ref{ex_pt}) and another one with pressure real analytic everywhere (Example \ref{ex:npt}). A  useful lemma in order to construct an example of a phase transition, is the following

\begin{lemma} \label{l:ex:2}
Let $(a_n)_n$ be a sequence of positive real numbers such that $\sum_{n=1}^{\infty}a_n^t$ converges for every $t>t^\ast$ and diverges at $t=t^\ast$. Then there exists a sequence $(\varepsilon_n)_n$ of positive numbers such that $\lim_{n\to\infty} \varepsilon_n = 0$ and
$$\sum_{n=1}^{\infty}a_n^{t^\ast+\varepsilon_n}<\infty.$$
\end{lemma}

\begin{proof}[Proof of Lemma \ref{l:ex:2}] Let $(\alpha_m)_m$ any sequence of real numbers in $(0,1]$ converging to zero. Note that, for every $m\geq 1$ we have
$$\sum_{n=1}^{\infty}a_n^{t^\ast+\alpha_m}<\infty.$$
Then, there exists an integer $N_m\geq 1$ such that $$\sum_{n=N_m}^{\infty}a_n^{t^\ast+\alpha_m}<1/m^2.$$
We can suppose without loss of generality that $N_m<N_{m+1}$. Define $\varepsilon_n$ for every $N_m\leq n < N_{m+1}$ as $\varepsilon_n=\alpha_m$, and $\varepsilon_n=1$ for $1\leq n < N_1$. Thus
\begin{eqnarray*}
\sum_{n=1}^{\infty} a_n^{t^\ast+\varepsilon_n} &=&  \sum_{n=1}^{N_{1}-1} a_n^{t^\ast+\varepsilon_n}+ \sum_{m=1}^{\infty} \sum_{n=N_m}^{N_{m+1}-1} a_n^{t^\ast+\varepsilon_n}\\
&=& \sum_{n=1}^{N_{1}-1} a_n^{t^\ast+1}+ \sum_{m=1}^{\infty} \sum_{n=N_m}^{N_{m+1}-1} a_n^{t^\ast+\alpha_m}\\
&\leq& \sum_{n=1}^{N_{1}-1} a_n^{t^\ast+1}+ \sum_{m=1}^{\infty} \sum_{n=N_m}^{\infty} a_n^{t^\ast+\alpha_m}\\
&\leq& \sum_{n=1}^{N_{1}-1} a_n^{t^\ast+1}+ \sum_{m=1}^{\infty} 1/m^2\\
&<& \infty.
\end{eqnarray*}
\end{proof}

\begin{example}(Phase transition) \label{ex_pt}  Let $\Gamma$ be a Schottky group with property $(\star)$ and assume that there are at least 2 different cusps, i.e. $N_2\ge 2$. Moreover assume there exists a unique parabolic generator $p$ with $\delta_{p}=\delta_{p,max}$. Recall that the series $\sum_{m\in\Z} e^{-\delta_p d(o,p^m o)}$ diverges since $p$ is a parabolic isometry of divergence type. Take a decreasing sequence of real numbers $\varepsilon_m>0$, as in Lemma \ref{l:ex:2}, such that $\lim_{m\to\infty}\varepsilon_m = 0$ and $\sum_{m\in\Z} e^{-(\delta_{p}+\varepsilon_m)d(o,p^m o)}<\infty$. Define a function $f^0:\Sigma\to\R^+$ by
\begin{enumerate}
\item $f^0(x)=\varepsilon_m\tau(x)$ if the first symbol of $x$ is $p^m$ for some $m\in\Z$.
\item $f^0(x)=1$ otherwise.
\end{enumerate}
Observe that since $\tau$ is locally H\"older, the function $f^0$ is also locally H\"older. We first see that $f^0\in \F$, for this it is enough to check that Condition (F2) holds. There exist constant $C$ independent of $m$ such that $d(o,p^m o)-C\le \tau(x)$ whenever $x\in C_{p^m}$. Then, if $x,y\in C_{p^m}$ we have $\tau(x)/\tau(y)\le d(o,p^m o)/ (d(o,p^m o)-C)$, i.e. $\sup_{x\in C_{p^m}}\tau(x)/\inf_{x\in C_{p^m}}\tau(x)$ is uniformly bounded in $m$, this implies Condition (F2).
 As shown in \cite[Section 2]{brw}, we can construct a continuous function $f:Y\to \R$ with $\Delta_{f}=f^0$. We define $t:\Sigma \to \R$ by $t(x)=(s_\infty+\varepsilon_m)$ if the first symbol of $x$ is $p^m$, and $s_\infty$ otherwise. By simplicity we will denote $s(a^m)=t(x)$ if the first symbol of $x$ is $a^m$. Following notations and ideas of the second part of the proof of Theorem \ref{s_inftyshottky}, we obtain,

\begin{eqnarray*}
P_{\sigma}(-\Delta_f-s_\infty\tau)&=& \lim_{n \to \infty} \frac{1}{n} \log \sum_{x:\sigma^{n}x=x} \exp \left(\sum_{i=0}^{n-1} -(\Delta_f(\sigma^ix)+s_\infty\tau(\sigma^{i}x))\right)  \chi_{C_{h_{1}}}(x)\\
&\leq& \lim_{n \to \infty} \frac{1}{n} \log \sum_{x:\sigma^{n}x=x} \exp \left(\sum_{i=0}^{n-1} -\tau(\sigma^i x)t(\sigma^i x)\right)  \chi_{C_{h_{1}}}(x)\\
&\leq& \lim_{n \to \infty} \frac{1}{n} \log \sum_{a_{1},...,a_{n}}\sum_{m_{1},...,m_{n}} \prod_{i=1}^{n}C^{s(a_{i}^{m_i})}e^{-s(a_{i}^{m_i}) d(o,a_{i}^{m_i}o )}\\
&\leq& \lim_{n \to \infty} \frac{1}{n} \log C^{n(s_\infty+1)}\sum_{a_{1},...,a_{n}}\sum_{m_{1},...,m_{n}} \prod_{i=1}^{n}e^{-s(a_{i}^{m_i}) d(o,a_{i}^{m_i}o )}\\
&=& \lim_{n \to \infty} \frac{1}{n} \log C^{n(s_\infty+1)}\left(\sum_{a\in \mathcal{A}}\sum_{m\in\Z} e^{-s(a^m) d(o,a^m o )} \right)^{n}\\
&=&  \log C^{s_\infty+1}\left(\sum_{a\in \mathcal{A}}\sum_{m\in\Z} e^{-s(a^m) d(o,a^m o )} \right).\\
\end{eqnarray*}
Observe that $\sum_m e^{-s(a^m) d(o,a^m o )}$ converges for every $s>\delta_{a}$ and every $a\neq p$. On the other hand, the series $\sum_{m\in\Z} e^{-(\delta_{p}+\varepsilon_m)d(o,p^m o)}$ is finite by construction. In particular $P_\sigma(-\Delta_f-s_\infty\tau)$ is finite. Observe that $f$ is a potential belonging to the family $\F$, then from Proposition \ref{phtr} it follows that $t\mapsto P_g(tf)$ exhibits a phase transition.\\
\end{example}

\begin{example}(No phase transition)  \label{ex:npt} Let $\Gamma$ be a Schottky group with property $(\star)$. Define $f^0:\Sigma\to\R^+$ to be constant of value 1 and construct a continuous function $f:Y\to \R$ with $\Delta_{f}=f^0$. Observe
$$P_{\sigma}(t-s_\infty\tau)=t+P_\sigma(-s_\infty\tau)=\infty.$$
Recall  that $P_\sigma(-s_\infty\tau)=\infty$, because the maximal parabolic generator is of divergence type (see the first part of Theorem \ref{s_inftyshottky}). Since $\tau$ is unbounded and $f^0$ is constant, we can apply Proposition \ref{phtr} to show that $t\mapsto P_g(tf)$ is real analytic in $\R$. In particular, it never attains the lower bound $s_\infty$.

\end{example}

\begin{remark}
In Example \ref{ex_pt} and Example \ref{ex:npt}, the potential $f$ is defined (a priori) only on the set $\Omega_0$. To extend it continuously to the entire manifold $T^1X/\Gamma$, it is enough to define it to be equal to $0$ on the complement $(T^1X/\Gamma) \setminus \Omega_0$.
\end{remark}

\end{document}